\documentclass[11pt]{article}
\usepackage{amsmath,amssymb,amsthm}
\usepackage{hyperref}
\usepackage{graphicx}
\usepackage[active]{srcltx}
\usepackage{indentfirst}
\usepackage{mathrsfs}
\usepackage{verbatim}
\usepackage{color}
\usepackage{caption}
\usepackage{subcaption}
\usepackage{epstopdf}
\graphicspath{{./pics/}}
\numberwithin{equation}{section}
\newtheorem{thm}{Theorem}[section]

\newtheorem{lem}{Lemma}[section]

\newtheorem{alg}{Algorithm}[section]

\theoremstyle{remark}
\newtheorem{example}{Example}[section]
\newtheorem{rem}{Remark}[section]

\numberwithin{equation}{section} 

\newcommand{\dx}[1]{\mathrm{d}#1}
\newcommand{\prnt}[1]{\left( #1 \right)}
\newcommand{\norm}[1]{\|#1\|}
\newcommand{\normL}[2]{\norm{#1}_{L^2\prnt{#2}}}
\newcommand{\normH}[2]{\norm{#1}_{H^1\prnt{#2}}}
\newcommand{\normLp}[3]{\norm{#1}_{L^{#3}\prnt{#2}}}
\newcommand{\const}[2]{C_{\text{#1},#2}}
\newcommand{\Const}[1]{C_{\text{#1}}}
\newcommand{\Csb}{C_{\text{stb}}}

\newcommand{\normg}{\|g\|_{L^2(\Gamma_A)}}

\DeclareMathOperator*{\argmin}{arg\,min}

\def\to{\rightarrow}
\def\Om{\Omega}  \def\om{\omega}

\def\m{\mbox}   
  
 \def\cT{\mathcal{T}}

\begin{document}
\title{Quasi-Optimality of an Adaptive Finite Element Method for Cathodic Protection}

\author{Guanglian Li\thanks{ Department of Mathematics, Imperial College London, London SW7 2AZ, UK. G. Li acknowledges the support from the Royal Society and Hausdorff Center for Mathematics, Bonn. G. Li also acknowledges the hospitality of HIM (trisemester program on multi-scale problems) and IPAM (long program on Computational Issues in Oil Field Applications). ({\tt guanglian.li@imperial.ac.uk, lotusli0707@gmail.com})}\and Yifeng Xu\thanks{Department of Mathematics and Scientific Computing Key
Laboratory of Shanghai Universities, Shanghai Normal University,
Shanghai 200234, China. The work of Y. Xu (corresponding author) was in part
supported by National Natural Science Foundation of China
(11201307), Ministry of Education of the People's Republic of
China through Specialized Research Fund for the Doctoral Program
of Higher Education (20123127120001) and Natural Science Foundation of Shanghai (17ZR1420800). ({\tt yfxu@shnu.edu.cn, yfxuma@aliyun.com})}}

\date{}

\maketitle

\begin{abstract}
In this work, we derive a reliable and efficient residual-typed error estimator for the finite element approximation of a
2d cathodic protection problem governed by a steady-state diffusion equation with a nonlinear boundary condition. We
propose a standard adaptive finite element method involving the D\"{o}rfler marking and a
minimal refinement without the interior node property. Furthermore, we establish the
contraction property of this adaptive algorithm in terms of the sum of the energy error and the scaled estimator. This essentially allows for a quasi-optimal convergence rate in terms of the number of elements over the underlying triangulation.
Numerical experiments are provided to confirm this quasi-optimality.
\end{abstract}


\vspace{0.2cm}


\section{Introduction}

Let $\Omega$ be a bounded polygonal domain in $\mathbb{R}^2$ with its boundary $\Gamma$ consisting of three mutually disjoint
parts: $\Gamma:=\Gamma_0\cup\Gamma_A\cup\Gamma_C$, all of which are line segments. This work is concerned with the numerical
treatment of the following problem:
\begin{equation}\label{diffequ}
-\nabla\cdot(\sigma \nabla u)=0\quad\mbox{in}~\Omega,
\end{equation}
\begin{equation}\label{bc}
\sigma\frac{\partial u}{\partial n}=0\quad\mbox{on}~\Gamma_0,\quad
\sigma\frac{\partial u}{\partial n}=g\quad\mbox{on}~\Gamma_A,\quad
\sigma\frac{\partial u}{\partial n}=-f(u)\quad\mbox{on}~\Gamma_C,
\end{equation}
where the conductivity $\sigma$ is assumed to be a piecewise $W^{1,\infty}$ function such that $\sigma_1\leq
\sigma\leq\sigma_2$ a.e. in $\Omega$ with two positive constants $\sigma_1$ and $\sigma_2$, ${n}$ is the unit
outward normal on $\Gamma$ and $g\in L^2(\Gamma_ A)$. The system \eqref{diffequ}--\eqref{bc} arises in
cathodic protection in electrochemistry. In a container $\Omega$ occupied by electrolyte, the first boundary
condition in \eqref{bc} describes insulation of the surface $\Gamma_0$ by painting. The second boundary equation
in \eqref{bc} reflects the fact that a current density $g$ on anodes $\Gamma_A$ induces an electrical potential
$u$ in $\Omega$ governed by \eqref{diffequ}. The corrosion process on cathodes $\Gamma_C$ is slowed down
through the nonlinear relation $f$, which depends on the electrode material and is given by \cite{hs} either
\begin{equation}\label{ef1}
f_1(u)=C_{1}u+C_{2}u^3
\end{equation}
or the Butler-Volmer function
\begin{equation}\label{ef2}
f_2(u)=C_{5}(e^{C_3 u}-e^{-C_4 u}),
\end{equation}
where $C_i$, $i=1\cdots 5$, are all positive constants.

In problem \eqref{diffequ}-\eqref{bc}, the sudden change of the boundary condition from Neumann type on anodes and
the insulated part to a nonlinear one on cathodes gives rise to local solution singularities in these regions.
Furthermore, internal layers may appear due to the discontinuity of the conductivity. Consequently, the computational
efficiency will be compromised if a uniform mesh refinement is employed in the finite element
discretization. One remedy in practice is to employ adaptivity techniques featuring local refinement so that
numerical results can attain better accuracy with minimum degrees of freedom. The aim of this work is to
investigate the computational complexity of an adaptive finite element method (AFEM) for problem \eqref{diffequ}-\eqref{bc}.

A typical adaptive algorithm comprises successive iterations of the following loop:
\begin{equation}\label{afem-loop}
\mbox{SOLVE}~\rightarrow~\mbox{ESTIMATE}~\rightarrow~\mbox{MARK}~\rightarrow~\mbox{REFINE}.
\end{equation}
That is, SOLVE yields a finite element approximation on the current mesh; ESTIMATE computes the relevant a
posteriori error estimator; MARK picks some elements to be subdivided; REFINE produces a new finer mesh.

The module ESTIMATE, depending on some computable quantities, i.e., the discrete solution, local mesh size and given
problem data, plays an indispensable role in \eqref{afem-loop}. Since the seminal work \cite{br}, a posteriori error
estimation for FEMs has been well understood in scientific computing and engineering \cite{ao,ver}. As
to the mathematical theory of AFEM, e.g., convergence and computational complexity,
there have been great developments (see the overview \cite{cfpp,nsv} and the references therein) over the past thirty years.
For linear elliptic problems, this issue has been investigated at depth \cite{bdd,ckns,dks,ffp,stev1}. Recently, the analysis has been extended to some nonlinear problems; see \cite{bdk,dk} for
$p$-Laplacian and \cite{gmz2,gmz3} for quasi-linear equations.

Recently, we \cite{LiXu} have proposed an AFEM of the form \eqref{afem-loop} for problem \eqref{diffequ}-\eqref{bc} and
proved its plain convergence, namely the $H^1$-norm error and the sequence of relevant estimators both go to zero as
the loop \eqref{afem-loop} proceeds. This work is a continuation of \cite{LiXu}, and it is devoted to the complexity of
the algorithm. In the AFEM, $\mbox{ESTIMATE}$, $\mbox{MARK}$ and $\mbox{REFINE}$ use a residual-type a posteriori error
estimator, D\"{o}rfler strategy and the bisection \cite{koss,stev2}, respectively.  The main contributions include
a contraction property in Theorem \ref{thm_conv} and the quasi-optimality computational complexity in terms of the
number of elements associated with underlying triangulations in Theorem \ref{thm_opt}.

Our analysis is inspired by \cite{ffp}, to first obtain optimal marking for the error estimator, cf. Lemma \ref{lem_optmar}
and then the optimal decay rate for the energy error plus an oscillation term so that the upper bound of the parameter in
D\"{o}rfler strategy is independent of the efficiency constant.
However, the nonlinear term $f$ on the boundary $\Gamma_C$ requires a different treatment.
First, due to the presence of the nonlinear term $f$ on the boundary $\Gamma_{C}$, the Galerkin orthogonality fails.
We employ the energy functional instead of the energy norm as in \cite{dk,gmz3}.
By the equivalence of $\mathcal{J}(u_\mathcal{T})-\mathcal{J}(u)$ and $\|u-u_{\mathcal{T}}\|_{H^1(\Om)}$, cf.
Lemma \ref{lem_equiv}, we prove that the adaptive algorithm reduces the sum of energy error and the scaled
estimator for any two consecutive iterations. Second, instead of standard arguments for linear problems,
we use the generalized H\"{o}lder inequality and the stability of solutions to establish a
C\'{e}a-type lemma for complexity estimate, cf. Lemma \ref{lem_cea}.

The remainder of the paper is organized as follows. Section 2 is devoted to the {\em a posteriori} error analysis. An adaptive algorithm
to approximate problem \eqref{diffequ}-\eqref{bc} is described in Section 3. We prove the convergence of this algorithm by a
contraction property in Section 5 after presenting preliminary results in Section 4. Section 6 focuses on the quasi-optimal
convergence rate. Throughout, we adopt standard notation for Sobolev spaces and related norms and semi-norms. Moreover,
any generic constant, with or without subscript, is independent of the mesh size and is not necessarily the same at each occurrence.

\section{A posteriori error analysis}
In this section, we shall derive a residual-type error estimator for the finite element approximation of problem \eqref{vp}, which
forms the basis of our AFEM. To introduce the AFEM, we first recall the variational formulation of problem \eqref{diffequ}-\eqref{bc}:
find $u\in H^1(\Omega)$ such that
\begin{equation}\label{vp}
\int_{\Omega}\sigma \nabla u\cdot{\nabla}v \;\dx{x}+\int_{\Gamma_C}f(u)v \; \dx{s}=\int_{\Gamma_A}g v \;\dx{s}\quad\text{ for all }~v\in H^1(\Omega).
\end{equation}
We refer to \cite{hs,ht} for its unique solvability.
For $f$ defined in \eqref{ef1} and \eqref{ef2}, it is easy to check that $f'$ is convex and there exists an $\alpha>0$ such that
\begin{equation}\label{prop_ef}
f'(t)\geq\alpha\quad\forall~t\in\mathbb{R},\quad \text{ and }\quad f(0)=0.
\end{equation}
Then by the application of the Poincar\'{e} inequality, the boundedness of $\sigma$ and the trace theorem, we arrive at
\begin{equation}\label{norm_equ}
\beta_1\|v\|_{H^1(\Omega)}^2\leq\int_\Om\sigma|{\nabla}v|^2
\dx{x}+\alpha\int_{\Gamma_C}v^2\dx{s}\leq\beta_2\|v\|_{H^1(\Omega)}^2\quad\text{ for all }~v\in
H^1(\Omega),
\end{equation}
where $\beta_1$ and $\beta_2$ are positive constants depending on $\sigma$, $\alpha$,
 $\Gamma_C$ and $\Om$.

Utilizing \eqref{vp} with $v:=u$, together with the mean value theorem, \eqref{prop_ef} and \eqref{norm_equ}, we can obtain
\begin{align*}
\beta_1\|u\|^2_{H^1(\Omega)}&\leq\int_{\Omega}\sigma|{\nabla}u|^2
\dx{x}+\alpha\int_{\Gamma_C}u^2\dx{s}\leq\int_{\Omega}\sigma|{\nabla}u|^2
\dx{x}+\int_{\Gamma_C}f(u)u \dx{s}\\
&=\int_{\Gamma_A}g u \dx{s}\leq\|g\|_{L^2(\Gamma_A)}\|u\|_{L^2(\Gamma_A)}.
\end{align*}
Then the trace theorem yields the {\em a priori} estimate to problem \eqref{vp}:
\begin{equation}\label{stab_cont}
\|u\|_{H^1(\Omega)}\leq \Csb\|g\|_{L^2(\Gamma_A)}
\end{equation}
with $\Csb>0$ being a constant depending on $\sigma$, $\alpha$, $\Gamma_C$ and $\Om$.

Furthermore, the continuity of the
imbedding $H^1(\Om)\hookrightarrow H^{\frac{1}{2}}(\Gamma)\hookrightarrow L^{q}(\Gamma)$ in 2d
for all $q<\infty$ implies the existence of a constant $\const{imb}{q}>0$ depending on $\Om$ and $q$, satisfying
\begin{align}\label{eq:imb1}
\normLp{v}{\Gamma_C}{q}\leq \const{imb}{q}\normH{v}{\Om} \text{ for all } v\in H^1(\Om).
\end{align}
Next, we proceed to the discretization. Let $\mathcal{T}$ be a shape-regular conforming triangulation of $\bar{\Omega}$
into a set of disjoint closed triangles such that the coefficient {$\sigma$ is piecewise $W^{1,\infty}$ over $\cT_0$. For each element $T\in\cT$, we denote its mesh size $h_{T}:=|T|^{\frac{1}{2}}$ and $\rho_{T}$ the diameter of the largest inscribed ball. We associate each triangulation $\cT$ with its shape regular parameter
$C_{\cT}:=\max_{T\in\cT}\frac{h_T}{\rho_T}$. Over the mesh $\mathcal{T}$, we consider the usual $H^1$-conforming finite element space $V_{\mathcal{T}}^{m}$, consisting of all piecewise polynomials of degree less than or equal to $m\in\mathbb{N}_{+}$, i.e.,
\[
V_\mathcal{T}^{m}:=\{v\in H^1(\Om)~|~v|_{T}\in P_m(T),\forall~T\in\mathcal{T}\}.
\]
Then the discrete problem corresponding to \eqref{vp} reads: find $u_{\mathcal{T}}\in V_{\mathcal{T}}^{m}$ such that
\begin{equation}\label{disvp}
\int_{\Omega}\sigma{\nabla}u_{\mathcal{T}}\cdot{\nabla}v_{\mathcal{T}} \dx{x}+\int_{\Gamma_C}f(u_{\mathcal{T}})v_{\mathcal{T}} \dx{s}=\int_{\Gamma_A}gv_{\mathcal{T}} \dx{s}\quad\text{ for all }~v_{\mathcal{T}}\in V_{\mathcal{T}}^{m}.
\end{equation}
Similar to the continuous case, the following stability estimate holds
\begin{equation}\label{stab_disc}
\|u_{\mathcal{T}}\|_{H^1(\Omega)}\leq \Csb\|g\|_{L^2(\Gamma_A)}.
\end{equation}

To describe the error estimator, we need a few notation and definitions. The collection of all edges (resp. all interior edges) in $\mathcal{T}$
is denoted by $\mathcal{F}_{\mathcal{T}}$ (resp. $\mathcal{F}_{\mathcal{T}}(\Omega)$) and its restriction on $\Gamma$ (resp. $\Gamma_{0}$,
$\Gamma_{A}$ and $\Gamma_{C}$) by $\mathcal{F}_{\mathcal{T}}(\Gamma)$ (resp. $\mathcal{F}_{\mathcal{T}}(\Gamma_{0})$, $\mathcal{F}_{\mathcal{T}}
(\Gamma_{A})$ and $\mathcal{F}_{\mathcal{T}}(\Gamma_{C})$).
The scalar $h_{F}:=|F|$ stands for the diameter of
$F\in\mathcal{F}_{\cT}$, which is
associated with a fixed normal unit vector ${n}_{F}$ in $\bar{\Omega}$ with ${n}_{F}={n}$ on the boundary $\partial\Omega$.
For each $T\in\mathcal{T}$, we denote $\om_{T}$ as the union of all elements in $\mathcal{T}$ with non-empty intersection with element $T$. For
any $F\in\mathcal{F}_{\mathcal{T}}$, $\om_{F}$ is the union of two elements that share $F$. Further, we let
\[
C_{\text{ov}}:=\max_{\tilde{T}\in \cT}\#\{ T\in\cT:\tilde{T}\subset \om_{T} \}.
\]
Let $I^{sz}_{\mathcal{T}}: H^1(\Om)\to V_{\cT}^{m}$ be the Scott-Zhang quasi-interpolation operator over $\mathcal{T}$ \cite{sz90}. Then for
all $v\in H^1(\Om)$, $T\in\cT$ and $F\in \partial T\cap\mathcal{F}_{\cT}$, there holds
\begin{align}\label{err:local}
h_{T}^{-\frac{1}{2}}\normL{v-I^{sz}_{\mathcal{T}}v}{F}+
h_{T}^{-1}\normL{v-I^{sz}_{\mathcal{T}}v}{T}\leq C_{I}\normL{{\nabla} v}{\om_{T}},
\end{align}
with $C_{I}$ a constant depending on the shape regularity parameter $C_{\cT}$.

For any $v_{\mathcal{T}}\in V_{\mathcal{T}}^{m}$, we define the residuals on each element  $T\in\mathcal{T}$ and each edge $F\in\mathcal{F}_{\cT}$ by
\[
R_{T}(v_{\mathcal{T}}):={\nabla}\cdot(\sigma{\nabla}v_{\mathcal{T}}),
\]
\[
J_{F}(v_{\mathcal{T}}):=\left\{\begin{array}{llll}
[\sigma{\nabla}v_{\mathcal{T}}\cdot{n}_{F}]\quad&
\m{for} ~~F\in\mathcal{F}_{\mathcal{T}}(\Omega),\\
\sigma{\nabla}v_{\mathcal{T}}\cdot{n}\quad&
\m{for} ~~F\in\mathcal{F}_{\mathcal{T}}(\Gamma_0),\\
g-\sigma{\nabla}v_{\mathcal{T}}\cdot{n}\quad& \m{for}~~F\in\mathcal{F}_{\mathcal{T}}(\Gamma_A),\\
f(v_{\mathcal{T}})+\sigma{\nabla}v_{\mathcal{T}}\cdot{n}\quad
&\m{for}~~F\in\mathcal{F}_{\mathcal{T}}(\Gamma_C),
\end{array}\right.
\]
where $[\cdot]$ denotes jumps across interior edges $F$:
\[
[v](x)=\lim_{t\to 0^{+}}v(x-t{n}_{F})-\lim_{t\to 0^{-}}v(x+t{n}_{F}).
\]
Then the local error indicator on any element $T\in\mathcal{T}$ is defined by
\begin{align} \label{localerr}
\eta_{\mathcal{T}}^2(v_{\mathcal{T}},T)&:
=h_{T}^2\|R_T(v_{\mathcal{T}})\|^2_{L^2(T)}
+\frac{1}{2}\sum_{F\in\partial T\cap\mathcal{F}_{\cT}(\Omega)}h_T\|J_{F}(v_{\mathcal{T}})\|_{L^2(F)}^2\\
&+\sum_{F\in\partial T\cap\mathcal{F}_{\cT}(\Gamma)}h_T\|J_{F}(v_{\mathcal{T}})\|_{L^2(F)}^2.\nonumber
\end{align}
The error estimator over the element patch $\mathcal{M}\subseteq\mathcal{T}$ is
\begin{equation*}
\eta_{\mathcal{T}}^2(v_{\mathcal{T}},\mathcal{M})
:=\sum_{T\in\mathcal{M}}\eta_{\mathcal{T}}^2(v_{\mathcal{T}},T).
\end{equation*}
Similarly, the oscillation term can be defined locally and globally by
\begin{alignat}{1}
\label{localosc}
\mathrm{osc}_{\mathcal{T}}^2(v_{\mathcal{T}},T):&=h_{T}^2\|R_T(v_{\mathcal{T}})-\bar{R}_{T}(v_{\mathcal{T}})\|_{L^2(T)}^2+
\sum_{F\in\partial T}h_{T}\|J_F(v_{\mathcal{T}})-\bar{J}_F(v_{\mathcal{T}})\|_{L^2(F)}^2,\\
\label{globalosc}
\mathrm{osc}_{\mathcal{T}}^2(v_{\mathcal{T}},\mathcal{M}):&=
\sum_{T\in\mathcal{M}}\mathrm{osc}^2_{\mathcal{T}}(v_{\mathcal{T}},T).
\end{alignat}
Here, $\bar{R}_T(v_\cT)$ is the integral average of ${R}_T(v_\cT)$ over $T$ if $m=1$, or the $L^2$-projection on $P_{m-2}(T)$ if $m\geq2$. $\bar{J}_F(v_{\mathcal{T}})$ is the $L^2$-projection of $J_F(v_{\mathcal{T}})$ on $P_{m-1}(F)$ if $F\in\mathcal{F}_\mathcal{T}\setminus\mathcal{F}_{\mathcal{T}}(\Gamma_C)$. When  $F\in\mathcal{F}_{\mathcal{T}}(\Gamma_C)$, $\bar{J}_F(v_{\mathcal{T}},g)$ is the $L^2$-projection on $P_{3m}(F)$
for $f$ in \eqref{ef1} and the $L^2$-projection on $P_{m-1}(F)$ for $f$ in \eqref{ef2}. If $\mathcal{M}=\mathcal{T}$, we simply write $\eta_{\mathcal{T}}(v_{\mathcal{T}})$ and $\mathrm{osc}_{\mathcal{T}}(v_{\mathcal{T}})$.

The following upper and lower bounds on the error estimator were given in \cite{LiXu}.
Here we give a more precise estimate with respect to the occurring constant.
For completeness, we provide the proof, since a related argument will be used in the proof of Lemma \ref{lem_disrel}.
\begin{thm}[Reliability]\label{thm_rel}
Let $u\in H^1(\Omega)$ and $u_{\mathcal{T}}\in V_{\mathcal{T}}^{m}$ be the solutions to problems \eqref{vp} and \eqref{disvp}, respectively.
Then there exists a positive constant $C_{rel}>0$ depending on $\sigma$, $\alpha$, $\Omega$, $\Gamma_C$, $m$ and $C_{\cT}$ such that
\begin{equation*}
\|u-u_{\mathcal{T}}\|_{H^1(\Omega)}^2\leq C_{\text{rel}}\;\eta^2_{\mathcal{T}}(u_\mathcal{T}).
\end{equation*}
\end{thm}
\begin{proof}
By \eqref{norm_equ}, \eqref{prop_ef}, mean value theorem and \eqref{vp} with $v:=u-u_{\mathcal{T}}$, we deduce
\begin{align}
\beta_1\|u-u_{\mathcal{T}}\|_{H^1(\Omega)}^2
&\leq \int_{\Omega}\sigma|{\nabla}(u-u_\mathcal{T})|^2\dx{x}+\alpha\int_{\Gamma_C}(u-u_{\mathcal{T}})^2\dx{s}\nonumber\\
&\leq \int_{\Omega}\sigma|{\nabla}(u-u_\mathcal{T})|^2\dx{x}+\int_{\Gamma_C}(f(u)-f(u_{\mathcal{T}}))
(u-u_{\mathcal{T}})\dx{s}\label{eq:galerkinO}\\
&=\int_{\Gamma_A}g v\dx{s}-\int_{\Omega}\sigma{\nabla}u_{\mathcal{T}}\cdot{\nabla}v \dx{x}-\int_{\Gamma_C}f(u_{\mathcal{T}})v \dx{s}.\nonumber
\end{align}
The discrete variational equation \eqref{disvp} implies
\[
\int_{\Gamma_A}g v_{\mathcal{T}}\dx{s}-\int_{\Omega}\sigma{\nabla}u_{\mathcal{T}}\cdot{\nabla}v_{\mathcal{T}} \dx{x}-\int_{\Gamma_C}f(u_{\mathcal{T}})v_{\mathcal{T}}\dx{s}=0 \quad\text{ for all }~v_{\mathcal{T}}\in V_{\mathcal{T}}.
\]
These two estimates together yield
\[
\forall~v_{\mathcal{T}}\in V_{\mathcal{T}}: \beta_1\|u-u_{\mathcal{T}}\|_{H^1(\Omega)}^2\leq \int_{\Gamma_A}g (v-v_{\mathcal{T}})\dx{s}-\int_{\Omega}\sigma{\nabla}u_{\mathcal{T}}\cdot{\nabla}(v-v_{\mathcal{T}}) \dx{x}-\int_{\Gamma_C}f(u_{\mathcal{T}})(v-v_{\mathcal{T}})\dx{s}.
\]
By elementwise integration by parts and taking $v_{\mathcal{T}}:=I^{sz}_{\mathcal{T}}v$, we obtain
\begin{align*}
\beta_1\|u&-u_{\mathcal{T}}\|_{H^1(\Omega)}^2
\leq\int_{\Gamma_A}g(v-v_{\mathcal{T}})\dx{s}-\int_{\Omega}\sigma{\nabla}u_{\mathcal{T}}\cdot
{\nabla}(v-v_{\mathcal{T}})\dx{x}
-\int_{\Gamma_C}f(u_{\mathcal{T}})(v-v_{\mathcal{T}})\dx{s} \\
&=\sum_{T\in\mathcal{T}}\Big(\int_T{\nabla}\cdot(\sigma{\nabla}u_{\mathcal{T}})(v-v_{\mathcal{T}})\dx{x}
-\frac{1}{2}\sum_{F\in\partial T\cap\mathcal{F}_{\mathcal{T}}(\Omega)}\int_{F}[\sigma{\nabla}u_{\mathcal{T}}\cdot{n}_{F}]
(v-v_{\mathcal{T}})\dx{s}\\
&\quad-\sum_{F\in\partial T\cap\mathcal{F}_{\mathcal{T}}(\Gamma_0)}\int_{F}\sigma{\nabla}u_{\mathcal{T}}\cdot{n}(v-v_{\mathcal{T}})\dx{s}
+\sum_{F\in\partial T\cap\mathcal{F}_{\mathcal{T}}(\Gamma_A)}\int_{F}(g-\sigma{\nabla}u_{\mathcal{T}}\cdot{n})(v-v_{\mathcal{T}})\dx{s}\\
&\quad-\sum_{F\in\partial T\cap\mathcal{F}_{\mathcal{T}}(\Gamma_C)}\int_{F}(f(u_{\mathcal{T}})+\sigma{\nabla}u_{\mathcal{T}}\cdot{n})
(v-v_{\mathcal{T}})\dx{s}\Big).
\end{align*}
Then a combination of \eqref{err:local} and Young's inequality completes the proof with $C_{\text{rel}}:={4C_{I}^{2}C_{\text{ov}}}/{\beta_1^{2}}$.
\end{proof}

\begin{thm}[Efficiency]\label{thm_eff}
Let $u\in H^1(\Omega)$ and $u_\cT\in V_\cT^{m}$ be the solutions
to problems \eqref{vp} and \eqref{disvp}, respectively. Then
there exists a positive constant $C_{eff}$ depending on $\sigma$, $\alpha$, $\Omega$, $\Gamma_C$, $m$ and $C_{\cT}$ such that
\begin{equation*}
C_{eff}\eta_{\mathcal{T}}^2(u_{\mathcal{T}})\leq
\|u-u_{\mathcal{T}}\|_{H^1(\Omega)}^2+\mathrm{osc}_{\mathcal{T}}^2(u_{\mathcal{T}}).
\end{equation*}
\end{thm}

\section{Adaptive algorithm}
Now we present the AFEM for problem \eqref{vp}. Let $\mathbb{T}$ be the set of all possible conforming triangulations of $\bar{\Omega}$
obtained from some initial mesh by successive bisections \cite{koss,mitc,stev2}. The refinement process ensures that all constant
depending on the shape regularity of $\mathcal{T}\in\mathbb{T}$ are uniformly bounded by a constant only depending on the
initial mesh \cite{nsv,traxler}. $\mathcal{T}_\ast$ is a called refinement of $\cT$ for $\cT\in\mathbb{T}$, if $\cT_\ast\in\mathbb{T}$
is produced from $\cT$ by a finite number of bisections.

The proposed adaptive algorithm is given below. For each triangulation $\cT_k$, $k\in \mathbb{N}_{0}$, we denote $V_{k}:=V_{\cT_{k}}^{m}$, $\eta_k:=\eta_{\cT_{k}}$ and $u_k:=u_{\cT_k}$.
\begin{alg}\label{afem}
Given an initial conforming mesh $\mathcal{T}_0$ and a parameter $\theta\in(0,1]$. Set $k:=0$.
\begin{enumerate}
\item\emph{(SOLVE)} Solve the discrete problem \eqref{disvp} on $\mathcal{T}_k$ for $u_{k}\in V_{k}$.
\item\emph{(ESTIMATE)} Compute the error estimator $\eta_{k}(u_k,g)$. 
\item\emph{(MARK)} Mark a subset $\mathcal{M}_k\subseteq\mathcal{T}_k$ with minimal cardinality such that
\begin{equation}\label{marking}
\eta_{k}^2(u_{k},\mathcal{M}_k)\geq\theta
\eta_{k}^2(u_{k}).
\end{equation}
\item\emph{(REFINE)} Refine each $T\in\mathcal{M}_k$ by bisection to get $\mathcal{T}_{k+1}$.
\item Set $k:=k+1$ and go to Step 1.
\end{enumerate}
\end{alg}

The convergence and quasi optimality of Algorithm \ref{afem} will be analyzed in
Sections \ref{sec:Convergence} and \ref{sec:opt}. A key ingredient is the so-called closure estimate
over the meshes $\{\cT_{k}\}_{k}$:
\begin{equation}\label{stevlemma}
\#\mathcal{T}_{k}\leq\#\mathcal{T}_{0}+C_{0}\sum_{j=0}^{k-1}\#\mathcal{M}_{j}
\end{equation}
with the constant $C_0$ depending on $C_{\cT_0}$ and $\#\cT$ denoting the number of elements in $\cT$.
This estimate was first proved in \cite[Theorem 2.4]{bdd} and then extended to the $n$-simplex case in \cite[Theorem 6.1]{stev2}.

\section{Auxiliary results}
This section is devoted to several technical lemmas for the convergence analysis of Algorithm \ref{afem}. As is well known, Galerkin orthogonality
or Pythagoras property is key to the convergence analysis of the linear problems, which regretfully fails for the nonlinear case. Thus a new
equivalent error has to be developed that can play the role of the Galerkin orthogonality property.

First, we introduce the associated functional to \eqref{vp} by
\begin{align}\label{eq:Jv}
\mathcal{J}(v):=\frac{1}{2}\int_{\Om}\sigma|{\nabla}v|^2\dx{x}
+\int_{\Gamma_C}F(v)\dx{s}-\int_{\Gamma_A}g v \;\dx{s}\quad\text{ for all }~v\in H^1(\Om)
\end{align}
with $F(t):=\int_{0}^{t}f(\tau)d\tau$.
Then problem \eqref{vp} is equivalent to the minimization problem \cite{LiXu}:
\begin{equation*}
u=\argmin\limits_{v\in H^{1}(\Om)}\mathcal{J}(v).
\end{equation*}
Let $u\in H^1(\Om)$ and $u_\mathcal{T}\in V_\mathcal{T}$ be solutions to problems \eqref{vp} and \eqref{disvp}, respectively. The non-negative quantity
\[
E(u_{\cT}):=\mathcal{J}(u_{\cT})-\mathcal{J}(u)
 \]
is referred to as the equivalent error throughout this paper.

The next lemma \cite[Lemma 2.1]{ht} is useful to handle exponential nonlinearity in \eqref{ef2}.
\begin{lem}\label{lem_ef2}
Let $v\in H^1(\Om)$ and $t>0$, then $e^{t|v|}\in L^1(\Gamma)$. Moreover,
there exists a positive constant $\Const{exp}$ independent of $v$, satisfying
\begin{equation*}
\int_\Gamma e^{t|v|}\dx{s}\leq 1+|\Gamma|+e^{\Const{exp}t^2\|v\|^2_{H^1(\Om)}}|\Gamma|<\infty.
\end{equation*}
Here, $|\Gamma|$ denotes the measure of $\Gamma$.
\end{lem}
Next we show the equivalence between $E(u_{\cT})$ and  $\|u-u_\mathcal{T}\|^2_{H^1(\Om)}$.
\begin{lem}\label{lem_equiv}
Let $u\in H^1(\Om)$ and $u_\mathcal{T}\in V_\mathcal{T}$ be solutions to problems \eqref{vp} and \eqref{disvp}, respectively.
Then there holds
\begin{equation*}
c_{equ}\|u-u_\mathcal{T}\|^2_{H^1(\Om)}\leq E(u_{\cT})
\leq C_{equ}\|u-u_\mathcal{T}\|^2_{H^1(\Om)}
\end{equation*}
with positive constants $c_{\text{equ}}$ and $C_{\text{equ}}$ depending on $\sigma$, $\alpha$, $\Om$, $\Gamma_C$, $m$ and $C_{\cT}$.
\end{lem}

\begin{proof}
Let $y(t):=\mathcal{J}(w(t))$ for $t\in [0,1]$, with $w(t):=(1-t)u+tu_\mathcal{T}$. Since $w(0)=u$ is the minimizer of $\mathcal{J}$
over $H^1(\Om)$, consequently, we can obtain $y'(0)=0$. In the meanwhile, we can obtain by Taylor's theorem that
\begin{align}\label{eq:666}
E(u_{\cT})=y(1)-y(0)=\int_{0}^1y''(t)(1-t)\mathrm{d}t.
\end{align}
To finish the proof, we need to compute $y''(t)$. In view that $F'(\cdot)=f(\cdot)$, an application of the chain rule implies
\[
\frac{\partial}{\partial t}F(w(t))=f(w(t))w'(t)=f(w(t))(u_\mathcal{T}-u),\quad
\frac{\partial^2}{\partial t^2}F(w(t))=f'(w(t))(u_\mathcal{T}-u)^2,
\]
which, together with the identity
\[
\frac{\partial^2}{\partial t^2}\Big(\frac{1}{2}\sigma|{\nabla}w(t)|^2\Big)=\sigma|{\nabla}(u_\mathcal{T}-u)|^2
\]
and the definition \eqref{eq:Jv},
yields
\begin{equation*}
y''(t)=\int_\Om\sigma|{\nabla}(u_\mathcal{T}-u)|^2\dx{x}
+\int_{\Gamma_C} f'(w(t))(u_\mathcal{T}-u)^2\dx{s}.
\end{equation*}
Combining with \eqref{eq:666}, we arrive at
\begin{equation}\label{lem_equiv1}
E(u_{\cT})=\int_0^1\int_\Om\sigma|{\nabla}(u_\mathcal{T}-u)|^2(1-t)\dx{x}\dx{t}
+\int_0^1\int_{\Gamma_C} f'(w(t))(u_\mathcal{T}-u)^2(1-t)\dx{s}\dx{t}.
\end{equation}
After combining with \eqref{prop_ef} and \eqref{norm_equ}, this derives the lower bound with $c_{\rm{equ}}:=\frac{1}{2}\beta_1$.

In the following we will prove the upper bound. The convexity of $f'$ implies
\begin{equation*}
f'(w(t))=f'((1-t)u+tu_{\mathcal{T}})\leq (1-t)f'(u)+tf'(u_\mathcal{T}) \text{ for all } t\in [0,1].
\end{equation*}
We will discuss the cases $f=f_1$ and $f=f_2$ separately.

For $f:=f_1$ in \eqref{ef1}, H\"{o}lder inequality, \eqref{eq:imb1}, \eqref{stab_cont} and \eqref{stab_disc} imply
\begin{align*}
\int_{\Gamma_C} f'(u)(u_\mathcal{T}-u)^2\dx{s}&=\quad C_1\int_{\Gamma_C}(u-u_\mathcal{T})^2\dx{s}+3C_2\int_{\Gamma_C}u^2(u-u_\mathcal{T})^2\dx{s}\nonumber\nonumber\\
&\leq C_1\normL{u-u_\mathcal{T}}{\Gamma_C}^2+3C_2\|u\|^2_{L^4(\Gamma_C)}\|u-u_\mathcal{T}\|^2_{L^4(\Gamma_C)}\nonumber\\
&\leq \Big(C_1\const{imb}{2}^2+3C_2\const{imb}{4}^4\Csb^2\normg^2\Big)\|u-u_\mathcal{T}\|^2_{H^1(\Om)}.
\end{align*}
Similarly, we obtain
\begin{align*}
&\int_{\Gamma_C} f'(u_{\cT})(u_\mathcal{T}-u)^2\dx{s}=C_1\int_{\Gamma_C}(u-u_\mathcal{T})^2\dx{s}
+3C_2\int_{\Gamma_C}u_\mathcal{T}^2(u-u_\mathcal{T})^2\dx{s}\nonumber\\
&\leq (C_1\const{imb}{2}^2+3C_2\const{imb}{4}^4\Csb^2\normg^2)\|u-u_\mathcal{T}\|^2_{H^1(\Om)}.
\end{align*}
Then the upper bound for $f:=f_1$ follows from these four estimates above.

In the case of \eqref{ef2}, $f:=f_2$. Then we can derive by application of the generalized H\"{o}lder inequality, together with Lemma \ref{lem_ef2}, \eqref{stab_cont} and \eqref{eq:imb1}, that
\begin{align*}
& \int_{\Gamma_C} f'(u)(u_\mathcal{T}-u)^2\dx{s}= C_3C_5\int_{\Gamma_C}e^{C_3u}(u-u_\mathcal{T})^2\dx{s}+C_4C_5\int_{\Gamma_C}e^{-C_4u}(u-u_\mathcal{T})^2{\rm d}s\nonumber\\
&\leq C_3C_5\|e^{C_3u}\|_{L^2(\Gamma_C)}\|u-u_\mathcal{T}\|_{L^4(\Gamma_C)}^2
+C_4C_5\|e^{-C_4u}\|_{L^2(\Gamma_C)}\|u-u_\mathcal{T}\|_{L^4(\Gamma_C)}^2\\
&:= C\|u-u_\mathcal{T}\|^2_{L^4(\Gamma_C)}\leq C\const{imb}{4}^2\normH{u-u_{\cT}}{\Om}^2.
\end{align*}
A similar argument yields
\begin{align*}
\int_{\Gamma_C} f'(u_\cT)(u_\mathcal{T}-u)^2\dx{s}\leq C\const{imb}{4}^2\normH{u-u_{\cT}}{\Om}^2.
\end{align*}
Finally, we can complete the proof after collecting these estimates above. 
\end{proof}
\begin{rem}\label{rem_equiv}
Note that let $\cT_{\ast}$ be a refinement of $\cT$ and $u_{\cT_\ast}$
be the solution to problem \eqref{disvp} over $\cT_\ast$. Then the estimate in Lemma \ref{lem_equiv} still holds should $u$ be replaced with $u_{\cT_\ast}$. Note also that when proving the contraction property of Algorithm \ref{afem} in Section \ref{sec:Convergence}, we will resort to the following identity
\[
\mathcal{J}(u_{\mathcal{T}_{\ast}})-\mathcal{J}(u)=
\mathcal{J}(u_{\cT})-\mathcal{J}(u)-(\mathcal{J}(u_{\cT})-\mathcal{J}(u_{\cT_\ast}))
\]
instead of Pythagoras property in the form of energy norm,
which is an important ingredient in relevant arguments for linear elliptic problems, but now fails for the nonlinear problem \eqref{vp}.
\end{rem}
The remaining of this section is concerned with the auxiliary results in the analysis: estimator reduction, C\'{e}a's lemma, oscillation perturbation and discrete reliability.
\begin{lem}[Estimator reduction]\label{lem_estred}
Let $\mathcal{T}\in\mathbb{T}$, $\mathcal{M}\subset\mathcal{T}$
and $\mathcal{T}_\ast\in\mathbb{T}$ be obtained from $\mathcal{T}$ by Algorithm \ref{afem} with $\mathcal{M}$ being the marked set. Let $u_\mathcal{T}\in V_\mathcal{T}$ and $u_{\mathcal{T}_\ast}\in V_{\mathcal{T}_\ast}$ be the solutions to problem \eqref{disvp} over $\mathcal{T}$ and $\mathcal{T}_\ast$ respectively. Then there exists a constant $\Const{est}$ depending only on $\sigma$, $\alpha$, $\Omega$, $\Gamma_C$, $C_{\cT}$, $m$ and $\|g\|_{L^2(\Gamma_C)}$ satisfying
\begin{equation*}
\forall \delta>0:\quad\eta_{\mathcal{T}_\ast}^2(u_{\mathcal{T}_\ast})\leq (1+\delta) \Big(\eta_{\mathcal{T}}^2(u_{\mathcal{T}})-\lambda\eta_{\mathcal{T}}^2
(u_{\mathcal{T}},\mathcal{M})\Big)+\Const{est}\|u_{\mathcal{T}_\ast}-
u_{\mathcal{T}}\|_{H^1(\Omega)}^2,
\end{equation*}
with $\lambda=1-\frac{1}{\sqrt{2}}$.
\end{lem}
Before proceeding to its proof, we need an auxiliary result:
\begin{lem}[Local perturbation of estimator] \label{lem_estredLocal}
Let $\cT\in \mathbb{T}$. Then there holds
\begin{align*}
\forall \delta\in (0,1):\quad\eta_{\cT}^2(v_{\cT},T)
&\leq (1+\delta)\eta_{\cT}^2(w_{\cT},T)
+(1+\frac{1}{\delta})\const{sym}{1}\Big(\normH{v_{\cT}-w_{\cT}}{\om_{T}}^2\nonumber\\
&+\sum_{F\in \partial T\cap\mathcal{F}_{\cT}(\Gamma_C)}h_{T}\normL{f(v_{\cT})-f(w_{\cT})}{F}^2\Big)
\end{align*}
for all $v_{\cT},w_{\cT}\in V_{\cT}^{m}$ and $T\in \cT$. Here, the constant $\const{sym}{1}$ depends on $C_{\cT}$, $m$, $\Om$, $\Gamma_C$, $\sigma$,  $\normH{v_{\cT}}{\Om}$ and $\normH{w_{\cT}}{\Om}$.
\end{lem}
\begin{proof}
Let $e:=v_{\cT}-w_{\cT}$. Firstly, we can obtain by the definition \eqref{localerr} combining with an application of the Young's inequality that
\begin{align*}
\forall \delta\in (0,1):\quad\eta_{\cT}^2(v_{\cT},T)
&\leq (1+\delta)\eta_{\cT}^2(w_{\cT},T)
+(1+\frac{1}{\delta})\Big(h_{T}^2\|\nabla\cdot(\sigma\nabla e)\|_{L^2(T)}^2\nonumber\\
&+\sum_{F\in \partial T}h_{T}\normL{J_F(v_{\cT})-J_F(w_{\cT})}{F}^{2}\Big).
\end{align*}
The inverse estimate indicates for some positive constant $\Const{inv}$ depending only on $\mathcal{C}_{\cT}$ and $m$, there holds
\begin{equation}\label{eq:inv}
\begin{aligned}
\|\nabla e\|_{L^2(T)}&\leq\Const{inv}h_{T}^{-1}\|e\|_{L^2(T)}\\
\|\nabla e\|_{L^2(F)}&\leq\Const{inv}h_{T}^{-1/2}\|\nabla e\|_{L^2(T)}.
\end{aligned}
\end{equation}
Combining those two estimate, we can obtain
\begin{align}\label{eq:aux1}
\forall \delta\in (0,1):\quad\eta_{\cT}^2(v_{\cT},T)
&\leq (1+\delta)\eta_{\cT}^2(w_{\cT},T)
+(1+\frac{1}{\delta})\Big(\const{aux}{1}\normH{e}{\om_{T}}^2\nonumber\\
&+\sum_{F\in \partial T\cap\mathcal{F}_{\cT}(\Gamma_C)}h_{T}\normL{f(v_{\cT})-f(w_{\cT})}{F}^2\Big).
\end{align}
Here, the constant $\const{aux}{1}$ depends on $C_{\cT}$, $m$ and $\sigma$. Therefore,
it suffices to bound the last term.

For $f:=f_1$ in \eqref{ef1}, a direct calculation leads to
\[
|f(v_{\cT})-f(w_{\cT})|^2\leq |e|^2(2C_1^2+36C_2^2(v_{\cT}^4+w_{\cT}^4)).
\]
The inverse estimate \cite[Section 4.5]{Brenner&Scott} gives
\begin{align}\label{eq:inv3}
\normLp{e}{F}{q}\leq \const{inv}{q}h_{T}^{1/q-1/2}\normLp{e}{F}{2} \text{ for } q>2
\end{align}
with the constant $\const{inv}{q}$ depending on $C_{\cT}$ and $m$.

The combination of these two inequalities with $q:=6$, together with the generalized H\"{o}lder's inequality and \eqref{eq:imb1}, yields
\begin{align*}
\normLp{f(v_{\cT})-f(w_{\cT})}{F}{2}^2&\leq \int_{F}|e|^2(2C_1^2+9C_2^2(v_{\cT}^4+w_{\cT}^4))\dx{s}\nonumber\\
&\leq\const{aux}{2}\normL{e}{F}^2, 
\end{align*}
with
$
\const{aux}{2}:=2C_1^2+36C_{2}^2\const{inv}{6}^2h_{T}^{-2/3}\const{imb}{6}^4(\normH{v_{\cT}}{\Om}^4+\normH{w_{\cT}}{\Om}^4).
$
Together with \eqref{eq:aux1}, this proves the desired result with
$\const{sym}{1}:=\max\{\const{aux}{1},h_{T}\const{aux}{2}\}$.

For $f:=f_2$ in \eqref{ef2}, note that the convexity of $f'(\cdot)$, together with the mean value theorem, implies
for some $s\in (0,1)$, there holds
\begin{align}
|f(v_{\cT})-f(w_\mathcal{T})|&=f'(s v_{\cT}+(1-s)w_\mathcal{T})|v_{\cT}-w_\mathcal{T}|\nonumber\\
&\leq(s f'(v_{\cT})+(1-s)f'(w_\mathcal{T}))|v_{\cT}-w_\mathcal{T}|.\label{eq:mean}
\end{align}
Consequently, taking square on both sides and employing the definition \eqref{ef2}, leads to
\[
|f(v_{\cT})-f(w_{\cT})|^2\leq 4C_{5}^2|e|^2\Big(C_3^2(e^{2C_{3}v_{\cT}}+e^{2C_{3}w_{\cT}})+C_4^2(e^{-2C_{4}v_{\cT}}+e^{-2C_{4}w_{\cT}})\Big).
\]
Then integrating both sides over $F$ and invoking the generalized H\"{o}lder's inequality, in combination with \eqref{eq:inv3} with $q:=4$ and Lemma \ref{lem_ef2}, we get
\begin{align}
\normLp{f(v_{\cT})-f(w_{\cT})}{F}{2}^2&\leq
4C_5^{2}\Big(C_3^{2}(\int_Fe^{4C_3 v_{\cT}}\dx{s})^{1/2}+C_3^{2}(\int_Fe^{4C_3 w_{\cT}}\dx{s})^{1/2}\nonumber\\
&+C_4^{2}(\int_Fe^{-4C_4 v_{\cT}}\dx{s})^{1/2}+C_4^{2}(\int_Fe^{-4C_4 w_{\cT}}\dx{s})^{1/2}   \Big)\|v_{\cT}-w_\mathcal{T}\|_{L^4(F)}^2\nonumber\\
&\leq \const{aux}{3}h_T^{-1/2}\|v_{\cT}-w_\mathcal{T}\|_{L^2(F)}^{2}. \label{eq:aux888}
\end{align}
Here,
\begin{align*}
\const{aux}{3}:&=4\const{inv}{4}^{2}C_5^{2}\Big(C_3^{2}(2+2 h_{F}^{1/2}+e^{8C_3^{2}\Const{exp}\normH{v_{\cT}}{\Om}^2}+
e^{8C_3^{2}\Const{exp}\normH{w_{\cT}}{\Om}^2})\\
&+C_4^{2}(2+2 h_{F}^{1/2}+e^{8C_4^{2}\Const{exp}\normH{v_{\cT}}{\Om}^2}+e^{8C_4^{2}\Const{exp}\normH{w_{\cT}}{\Om}^2})\Big).
\end{align*}
Therefore, noticing \eqref{eq:aux1}, the assertion follows with
$\const{sym}{1}:=\max\{\const{aux}{1},h_{T}^{1/2}\const{aux}{3}\}$.
\end{proof}

\begin{proof} [Proof of Lemma \ref{lem_estred}] Let $T\in \cT_{\ast}$. Utilizing Lemma \ref{lem_estredLocal} with $\cT:=\cT^{*}$ and $v_{\cT}:=u_{\cT}$,
$w_{\cT}:=u_{\cT_{\ast}}$, we can derive
\begin{equation}\label{lem_estred1}
\begin{split}
\eta_{\mathcal{T}_\ast}^2(u_{\mathcal{T}_\ast},T)&\leq
(1+\delta)\eta_{\mathcal{T}_\ast}^2(u_\mathcal{T},T)+(1+\frac{1}{\delta})\const{sym}{1}
\Big(\|u_{\mathcal{T}_\ast}-u_\mathcal{T}\|^2_{H^1(\omega_T)}\\
&\quad+\sum_{F\in\partial T\cap\mathcal{F}_{\cT_{\ast}}(\Gamma_C)}\|u_{\mathcal{T}_\ast}-u_{\mathcal{T}}\|^2_{L^2(F)}\Big).
\end{split}
\end{equation}
Summing over all elements $T\in\cT_{\ast}$ leads to
\begin{align*}
\eta_{\mathcal{T}_\ast}^2(u_{\mathcal{T}_\ast})&\leq
(1+\delta)\eta_{\mathcal{T}_\ast}^2(u_\mathcal{T})+
\const{sym}{1}(1+\frac{1}{\delta})\Big(
C_{\text{ov}}\|u_{\mathcal{T}_\ast}-u_\mathcal{T}\|^2_{H^1(\Om)}\\
&\quad+
\|u_{\mathcal{T}_\ast}-u_{\mathcal{T}}\|^2_{L^2(\Gamma_C)}\Big).
\end{align*}
Now the Sobolev imbedding theorem \eqref{eq:imb1} implies
\begin{align}\label{eq:222}
\eta_{\mathcal{T}_\ast}^2(u_{\mathcal{T}_\ast})\leq (1+\delta)\eta_{\mathcal{T}_\ast}^2(u_\mathcal{T})+C_{\text{est}}\|u_{\mathcal{T}_\ast}-u_\mathcal{T}\|^2_{H^1(\Om)}
\end{align}
with $C_{\text{est}}:=\const{sym}{1}(1+\frac{1}{\delta})
( C_{\text{ov}}+\const{imb}{2}^2)$. Note that
thanks to the stability estimate \eqref{stab_disc} for $u_\cT$ and $u_{\cT_\ast}$, the positive constant $C_{\text{est}}$ depends only on $\sigma$, $\alpha$, $\Omega$, $\Gamma_C$, $C_{\cT}$, $m$ and $\|g\|_{L^2(\Gamma_C)}$. Note also that
\begin{align}\label{eq:222-1}
\eta_{\mathcal{T}_\ast}^2(u_\mathcal{T})
=\eta_{\mathcal{T}}^2(u_\mathcal{T},\cT\cap\cT_{*} )+\eta_{\mathcal{T}_\ast}^2(u_\mathcal{T},\cT_{*}\backslash \cT).
\end{align}
In view that for each element $T\in\cT_{*}\backslash \cT$, there exists a unique $\hat{T}\in \cT\backslash \cT_{*}$,
s.t., $T\subset\hat{T}$ and $h_{T}\leq \frac{1}{\sqrt{2}}h_{\hat{T}}$.
Then by definition \eqref{localerr}, we arrive at
\begin{align}\label{eq:222-2}
\eta_{\mathcal{T}_\ast}^2(u_\mathcal{T},\cT_{*}\backslash \cT)\leq \frac{1}{\sqrt{2}}
\eta_{\mathcal{T}}^2(u_\mathcal{T},\cT\backslash \cT_{*}).
\end{align}
Finally, the combination of \eqref{eq:222}, \eqref{eq:222-1}, \eqref{eq:222-2} and $\mathcal{M}\subseteq\cT\backslash \cT_{*}$ yields the assertion.
\end{proof}

\begin{rem}\label{rem_estredLocalper}
     If we exchange $u_{\cT}$ and $u_{\cT_{*}}$ in \eqref{lem_estred1} and then sum up it over $\cT\cap\cT_{*}$ instead, a similar argument leads to
\begin{align}\label{eq:777}
\eta_{\mathcal{T}}^2(u_{\mathcal{T}},\cT\cap\cT_{*})
\leq (1+\delta)\eta_{\mathcal{T}_{*}}^2(u_{\mathcal{T}_*},\cT\cap\cT_{*})
+C_{\text{est}}\|u_{\mathcal{T}_\ast}-u_\mathcal{T}\|^2_{H^1(\Om)}.
\end{align}
\end{rem}

\begin{lem}[C\'{e}a's lemma]\label{lem_cea}
Let $u$ and $u_\mathcal{T}$ be solutions to problems \eqref{vp} and \eqref{disvp}
over some mesh $\mathcal{T}\in\mathbb{T}$. Then
\begin{equation*}
\|u-u_\mathcal{T}\|^2_{H^1(\Om)}\leq C_{cea}\inf_{v_\mathcal{T}\in V_\mathcal{T}}\|u-v_\mathcal{T}\|^2_{H^1{(\Om)}}.
\end{equation*}
\end{lem}
\begin{proof}
We derive from \eqref{eq:galerkinO} and the Galerkin Orthogonality that for any $v_\mathcal{T}\in V_\mathcal{T}$, there holds
\begin{align*}
\beta_1\|u-u_\mathcal{T}\|^2_{H^1(\Om)}
&\leq\int_\Om\sigma|{\nabla}(u-u_\mathcal{T})|^2 \dx{x}+
\int_{\Gamma_C}(f(u)-f(u_\mathcal{T}))(u-u_\mathcal{T})\dx{s}\nonumber\\
&= \int_\Om\sigma{\nabla}(u-u_\mathcal{T})\cdot{\nabla}(u-v_\mathcal{T}) \dx{x}+
\int_{\Gamma_C}(f(u)-f(u_\mathcal{T}))(u-v_\mathcal{T})\dx{s}.
\end{align*}
We focus on the second term, and claim
\begin{align*}
\Big|\int_{\Gamma_C}(f(u)&-f(u_\mathcal{T}))(u-v_\mathcal{T})\dx{s}\Big|\leq \const{aux}{4}\normH{u-u_\mathcal{T}}{\Om}\normH{u-v_\mathcal{T}}{\Om}.
\end{align*}
For $f:=f_1$ in \eqref{ef1}, notice that
\[
|f(u)-f(u_\mathcal{T})|\leq |u-u_{\cT}|(C_1+3C_{2}(u^2+u_{\cT}^2)) \text{ on }\Gamma_C,
\]
therefore, an application of the generalized H\"{o}lder's inequality implies
\begin{align*}
\Big|\int_{\Gamma_C}(f(u)&-f(u_\mathcal{T}))(u-v_\mathcal{T})\dx{s}\Big|\leq C_{1}\normL{u-u_\mathcal{T}}{\Gamma_C}\normL{u-v_\mathcal{T}}{\Gamma_C}\\
&+3C_{2}\normLp{u-u_\mathcal{T}}{\Gamma_C}{4}
\normLp{u-v_\mathcal{T}}{\Gamma_C}{4}\Big(\normLp{u}{\Gamma_C}{4}^2
+\normLp{u_{\cT}}{\Gamma_C}{4}^2\Big).
\end{align*}
Thus the claim follows from \eqref{eq:imb1}, \eqref{stab_cont} and \eqref{stab_disc}
with
$
\const{aux}{4}:=C_{1}\const{imb}{2}^2+6C_{2}\Csb^2\const{imb}{4}^4\normg^{2}.
$
For $f:=f_2$ in \eqref{ef2}, by \eqref{eq:mean}, we deduce
\[
|f(u)-f(u_\mathcal{T})|\leq C_{5}|u-u_{\cT}|\Big(C_{3}e^{C_{3}u}+C_{4}e^{-C_{4}u}+C_{3}e^{C_{3}u_{\cT}}+C_{4}e^{-C_{4}u_{\cT}}\Big) \text{ on }\Gamma_{C}.
\]
Multiplying by $u-v_{\cT}$, integrating over $\Gamma_C$ and applying generalized H\"{o}lder's inequality and Lemma \ref{lem_ef2} yield
\begin{align*}
\Big|\int_{\Gamma_C}(f(u)&-f(u_\mathcal{T}))(u-v_\mathcal{T})\dx{s}\Big|\leq
C_{5}\normLp{u-u_{\cT}}{\Gamma_C}{4}\normLp{u-v_{\cT}}{\Gamma_C}{4}\Big(  C_{3}\big(2+2|\Gamma_C|^{1/2}\\
&+(e^{2C_{3}^{2}\Const{exp}\normH{u}{\Om}^{2}}
+e^{2C_{3}^{2}\Const{exp}\normH{u_{\cT}}{\Om}^{2}})|\Gamma_C|^{1/2}\big)+C_{4}\big(2+2|\Gamma_C|^{1/2}\\
&+(e^{2C_{4}^{2}\Const{exp}\normH{u}{\Om}^{2}}+e^{2C_{4}^{2}\Const{exp}\normH{u_{\cT}}{\Om}^{2}})|\Gamma_C|^{1/2}\big)\Big).
\end{align*}
Thereafter, we establish the claim by appealing to \eqref{eq:imb1}, \eqref{stab_cont} and \eqref{stab_disc} with
\begin{align*}
\const{aux}{4}:&=C_{5}\const{imb}{4}^{2}\Big(  C_{3}\big(2+2|\Gamma_C|^{1/2}+2|\Gamma_C|^{1/2}e^{2C_{3}^{2}\Const{exp}\Csb^2\normg^2}|\Gamma_C|^{1/2}\big)\\
&+C_{4}\big(2+2|\Gamma_C|^{1/2}
+2|\Gamma_C|^{1/2}e^{2C_{4}^{2}\Const{exp}\Csb^2\normg^{2}}\big)\Big).
\end{align*}
The desired assertion follows from the claim and Young's inequality with
$\Const{cea}:=2({\sigma_2^{2}}/{\beta_1^2}+{\const{aux}{4}^{2}}/{\beta_1^2}).$
\end{proof}
To establish the oscillation perturbation estimate, we may
follow the proof of Lemma \ref{lem_estredLocal} to obtain the local
perturbation of oscillation. However, this leads to the issue that the related constant depends on $v_\cT$ and $w_\cT$ as $C_{\mathrm{sym},1}$ in Lemma \ref{lem_estredLocal}. As a result, the constant in the subsequent convergence rate involves the finite
element function $v_\cT$, which should be avoided in order to establish the quasi-optimality estimate. To this end, we keep the nonlinear function in the following estimate.
\begin{lem}[Oscillation perturbation]\label{lem_osc_p}
Let $\mathcal{T}, \mathcal{T}_\ast\in\mathbb{T}$ with $\mathcal{T}_\ast$ being a refinement of $\mathcal{T}$.
If $v_\cT\in V_\mathcal{T}$ and $v_{\mathcal{T}_\ast}\in V_{\mathcal{T}_\ast}$, then
\begin{equation*}
\mathrm{osc}^2_\cT(v_{\cT},\mathcal{T}\cap\mathcal{T}_\ast)\leq 2\mathrm{osc}^2_{\cT_\ast}(v_{\cT_\ast},\mathcal{T}\cap\mathcal{T}_\ast)+\Const{op}
\Big(\|v_{\cT_\ast}-v_\cT\|^2_{H^1(\Om)}
+\sum_{F\in\mathcal{F}_{\cT}(\Gamma_C)}h_{F}\|\tilde{f}(v_{\cT_{\ast}})
-\tilde{f}(v_{\cT})\|_{L^2(F)}^2\Big),
\end{equation*}
where $\Const{op}$ depends on $C_\cT$, $m$, $\sigma$ and $\Const{ov}$, and
\begin{equation}\label{eftilde}
    \tilde{f}:=\left\{
    \begin{array}{ll}
        0&\mbox{if}~f=f_1\\
        f_2&\mbox{if}~f=f_2.
    \end{array}
    \right.
\end{equation}
\end{lem}
\begin{proof}
Let $T\in\cT\cap\cT_{\ast}$ and denote $e:=v_\cT-v_{\cT_*}$. We can obtain from the definition \eqref{localosc} and the Young's inequality that
\begin{align}\label{localosc_per}
\forall \delta>0:\mathrm{osc}_{\cT}^2(v_{\cT},T)&\leq (1+\delta)\mathrm{osc}_{\cT}^2(v_{\cT_{*}},T)
    +(1+\frac{1}{\delta})\Big(h_{T}^2\|R_T(e)-\bar{R}_T(e)\|_{L^2(T)}^2\nonumber\\
    &+\sum_{F\in\partial T}h_T\|(J_{F}(v_{\cT})-J_F(v_{\cT_{*}}))
    -(\bar{J}_{F}(v_{\cT})-\bar{J}_F(v_{\cT_{*}}))\|_{L^2(F)}^2\Big).
\end{align}
Together with the inverse estimate \eqref{eq:inv}, we arrive at
\begin{align*}
h_{T}^2\|R_T(e)-\bar{R}_T(e)\|_{L^2(T)}^2&+\sum_{F\in\partial T\backslash\Gamma_C}h_T\|(J_{F}(v_{\cT})-J_F(v_{\cT_{*}}))
    -(\bar{J}_{F}(v_{\cT})-\bar{J}_F(v_{\cT_{*}}))\|_{L^2(F)}^2\\
    &\leq C_{\mathrm{aux},5}\|e\|_{H^1(\omega_T)}^2,
\end{align*}
where $C_{\mathrm{aux},5}$ depends on $C_\cT$, $m$ and $\sigma$.

Let $f=f_1$. Note that since $J_F(v)\in P_{3m}(\partial T\cap\Gamma_C)$ for all $v\in V_{\cT}$ and $T\in \cT$, and since $\bar{J}_F(v)$ is the $L^2$-projection of $J_F(v)$ on $P_{3m}(F)$, therefore, we can obtain
\[
\forall v\in V_{\cT} \text{ and } F\in \partial T\cap\Gamma_C:\quad J_F(v)-\bar{J}_F(v)=0.
\]
Furthermore, let $f=f_2$. In view that $\bar{J}_F(v)$ is the $L^2$-projection of $J_F(v)$ onto $P_{m-1}(F)$, we derive
\[
\forall F\in \partial T\cap\Gamma_C:\quad
\|(J_{F}(v_{\cT})-J_F(v_{\cT_{*}}))
    -(\bar{J}_{F}(v_{\cT})-\bar{J}_F(v_{\cT_{*}}))\|_{L^2(F)}\leq \|J_{F}(v_{\cT})-J_F(v_{\cT_{*}})\|_{L^2(F)}.
\]
Plugging those estimates above into \eqref{localosc_per} leads to
\begin{align}\label{eq:333}
\forall \delta>0:\mathrm{osc}_{\cT}^2(v_{\cT},T)&\leq (1+\delta)\mathrm{osc}_{\cT}^2(v_{\cT_{*}},T)
    +(1+\frac{1}{\delta})\Big(C_{\mathrm{aux},5}\|e\|_{H^1(\omega_T)}^2\nonumber\\
    &+2\sum_{F\in\partial T\cap\Gamma_C}h_T\|\tilde{f}(v_{\cT})-\tilde{f}(v_{\cT_{*}})\|_{L^2(F)}^2\Big).
\end{align}
In the meanwhile, note that $\mathrm{osc}_{\cT_{\ast}}(v_{\cT},T)=\mathrm{osc}_{\cT}(v_{\cT},T)$ and $v_{\cT}\in V_{\cT_{*}}$, then summing over $T\in \cT\cap\cT_{\ast}$ in \eqref{eq:333} leads to
\begin{align*}
\mathrm{osc}_{\cT}^2(v_{\cT},\cT\cap\cT_{\ast})\leq
(1+\delta)\mathrm{osc}_{\cT_{\ast}}^2(v_{\cT_{\ast}},\cT\cap\cT_{\ast})&+(1+\frac{1}{\delta})\Big(C_{\rm{aux},5}
C_{\text{ov}}\normH{v_{\cT_{\ast}}-v_{\cT}}{\Om}^2\\
&+2\sum_{F\in\mathcal{F}_{\cT}(\Gamma_C)}h_{F}\|\tilde{f}(v_{\cT_{\ast}})-\tilde{f}(v_{\cT})\|_{L^2(F)}^2\Big).
\end{align*}
The desired result follows by taking
$\delta:=1$ and $
\Const{op}:=2\times\max\{C_{\rm{aux},5}
C_{\text{ov}},2\}.
$
\end{proof}

\begin{lem}[Discrete reliability]\label{lem_disrel}
Let $\mathcal{T}, \mathcal{T}_\ast\in\mathbb{T}$ with $\mathcal{T}_\ast$ being a refinement of $\mathcal{T}$ and let $u_\mathcal{T}\in V_\mathcal{T}$, $u_{\mathcal{T}_\ast}\in
V_{\mathcal{T}_\ast}$ be solutions to problem \eqref{disvp} over $\mathcal{T}$ and $\mathcal{T}_\ast$,  respectively. Then there exists $C_{drel}>0$ depending only on $\sigma$ $\alpha$, $\Omega$, $\Gamma_C$, $m$ and $C_{\mathcal{T}}$ such that
\begin{equation*}
\|u_{\mathcal{T}_\ast}-u_\mathcal{T}\|_{H^1(\Om)}^2\leq \Const{drel}\eta^2_{\mathcal{T}}(u_\mathcal{T},\mathcal{T}\setminus\mathcal{T}_\ast).
\end{equation*}
\end{lem}
\begin{proof}
Using the operator $I_\mathcal{T}^{sz}$ \eqref{err:local} to $v:=u_{\mathcal{T}_\ast}-u_\mathcal{T}$ and noting $v=I_\mathcal{T}^{sz}v$ on
unrefined elements in $\mathcal{T}\cap\mathcal{T}_\ast$, the argument of Theorem \ref{thm_rel} completes the proof with $\Const{drel}:=\Const{rel}$.
\end{proof}

\section{Convergence}\label{sec:Convergence}

Now we show each iteration of Algorithm \ref{afem}  reduces the sum of the equivalent error and the scaled estimator, which  implies the convergence of the algorithm.
\begin{thm}[Contraction Property]\label{thm_conv}
Let $u\in H^1(\Om)$ be the solution to problem \eqref{vp} and $\{\mathcal{T}_k,V_k,u_k\}$ be a sequence of meshes, finite element
spaces and discrete solutions by Algorithm \ref{afem}. Then there exist constants $0<\mu<1$ and $\beta>0$
depending on $C_{\mathcal{T}_0}$ and $\theta$ such that
\begin{equation*}
E(u_{k+1})+\beta\eta_{k+1}^2(u_{k+1})\leq\mu(E (u_{k})+\beta\eta_{k}^2(u_k)).
\end{equation*}
\end{thm}
\begin{proof}
By the equality
$
E(u_{k+1}):=\mathcal{J}(u_{k+1})-\mathcal{J}(u)=
\mathcal{J}(u_{k})-\mathcal{J}(u)-(\mathcal{J}(u_{k})-\mathcal{J}(u_{k+1}))
$ and Lemma \ref{lem_estred} with $\mathcal{T}=\mathcal{T}_k$ and $\mathcal{T}_\ast=\mathcal{T}_{k+1}$, we can obtain for all $\beta>0$ that
\[
\begin{split}
E(u_{k+1})+\beta\eta_{k+1}^2(u_{k+1})&\leq
E(u_{k})+\beta(1+\delta) (\eta_{k}^2(u_{k})-\lambda\eta_{k}^2(u_{k},\mathcal{M}_k))\\
&\quad-(\mathcal{J}(u_{k})-\mathcal{J}(u_{k+1}))+\beta \Const{est}\|u_{k+1}-u_{k}\|_{H^1(\Omega)}^2.
\end{split}
\]
Then by taking $\beta:=c_{\text{equ}}/C_{\text{est}}$, an application of Lemma \ref{lem_equiv} and Remark \ref{rem_equiv} leads to
\[
E(u_{k+1})+\beta\eta_{k+1}^2(u_{k+1})\leq
E(u_{k})+\beta(1+\delta) \Big(\eta_{k}^2(u_{k})-\lambda\eta_{k}^2(u_{k},\mathcal{M}_k)\Big),
\]
which, together with the marking strategy \eqref{marking}, implies
\[
\begin{split}
E(u_{k+1})+\beta\eta_{k+1}^2(u_{k+1})&\leq
E(u_{k})-\beta(1+\delta)\lambda\frac{\theta}{2}\eta_{k}^2(u_{k})\\
&\quad+\beta(1+\delta)(1-\lambda\frac{\theta}{2})\eta_{k}^2(u_{k}).
\end{split}
\]
Now by Theorem \ref{thm_rel}, Lemma \ref{lem_equiv} and the choice $\beta=c_{\text{equ}}/C_{\text{est}}$, we obtain
\begin{align*}
E(u_{k+1})+\beta\eta_{k+1}^2(u_{k+1})\leq\mu_1(\delta)
E(u_{k})+\mu_2(\delta)\eta_{k}^2(u_{k})
\end{align*}
with
$
\mu_1(\delta):=1-\frac{(1+\delta) c_{\text{equ}}\lambda\theta}{2\Const{equ}\Const{rel}\Const{est}}\text{ and }\mu_2(\delta):=(1+\delta)(1-\lambda\frac{\theta}{2}).
$
The proof is completed by choosing $\delta>0$ small enough such that
\(
0<\mu:=\max(\mu_1(\delta),\mu_2(\delta))<1.
\)
\end{proof}

\section{Quasi-optimality}\label{sec:opt}

Now we give a quasi-optimal convergence rate for Algorithm \ref{afem}. We begin with a generalization of Cea's lemma in Lemma \ref{lem_cea}.
\begin{lem}\label{lem:lower-error}
Let $u$ and $u_\mathcal{T}$ be solutions to problems \eqref{vp} and \eqref{disvp}
over some mesh $\mathcal{T}\in\mathbb{T}$. Then
\begin{equation}\label{opt_total}
    \begin{aligned}
\|u-u_\mathcal{T}\|^2_{H^1(\Om)}+\mathrm{osc}_{\cT}^2(u_\cT)&\leq \Const{qs}\inf_{v_\cT\in V_\cT}\Big(\|u-v_\cT\|^2_{H^1(\Om)}+\mathrm{osc}_\cT^2(v_\cT)\\
&\qquad\qquad+\sum_{F\in\mathcal{F}_\cT(\Gamma_C)}h_F\|\tilde{f}(u)-\tilde{f}(v_\cT)\|^2_{L^2(F)}\Big).
\end{aligned}
\end{equation}
Here, the positive constant $\Const{qs}$ depends on $C_\cT$, $m$, $\sigma$, $g$ and $\Const{ov}$.
\end{lem}
\begin{proof}
Given $v_{\cT}\in V_{\cT}$. An application of Lemma \ref{lem_osc_p} with $v_\cT:=u_\cT$, $\mathcal{T}_\ast:=\cT$ yields
\begin{align*}
\mathrm{osc}_{\cT}^2(u_\cT)\leq 2\text{osc}_{\cT}^2(v_\cT)+\Const{op}\Big(\|u_{\cT}-v_\mathcal{T}\|^2_{H^1(\Om)}
+\sum_{F\in\mathcal{F}_\cT(\Gamma_C)}h_F
\|\tilde{f}(u_\cT)-\tilde{f}(v_\cT)\|^2_{L^2(F)}\Big).
\end{align*}
Then an application of the triangle inequality yields
\begin{align*}
\mathrm{osc}_{\cT}^2(u_\cT)&\leq 2\text{osc}_{\cT}^2(v_\cT)+2\Const{op}\Big(\|u-v_\mathcal{T}\|^2_{H^1(\Om)}+\|u-u_\mathcal{T}\|^2_{H^1(\Om)}\nonumber\\
&+\sum_{F\in\mathcal{F}_\cT(\Gamma_C)}h_F\|\tilde{f}(u)-\tilde{f}(u_\cT)\|^2_{L^2(F)}+
\sum_{F\in\mathcal{F}_\cT(\Gamma_C)}h_F\|\tilde{f}(u)-\tilde{f}(v_\cT)\|^2_{L^2(F)}\Big).
\end{align*}
Combining with Lemma \ref{lem_cea}, we arrive at
\begin{align}
\|u-u_\mathcal{T}\|^2_{H^1(\Om)}&+\mathrm{osc}_{\cT}^2(u_\cT)
\leq 2\text{osc}_{\cT}^2(v_\cT)+\Big(2\Const{op}+2\Const{op}
\Const{cea}+\Const{cea}\Big)\|u-v_\mathcal{T}\|^2_{H^1(\Om)}\label{eq:444}\\
&+2\Const{op}\Big(\sum_{F\in\mathcal{F}_\cT(\Gamma_C)}h_F\|\tilde{f}(u)-\tilde{f}(u_\cT)\|^2_{L^2(F)}+
\sum_{F\in\mathcal{F}_\cT(\Gamma_C)}h_F\|\tilde{f}(u)-\tilde{f}(v_\cT)\|^2_{L^2(F)}\Big).\nonumber
\end{align}
Let $f:=f_1$. Then $\tilde{f}=0$ by \eqref{eftilde}.
This proves the assertion by taking $\Const{qs}:=\max\{2, 2\Const{op}+2\Const{op}\Const{cea}+\Const{cea}\}$.

Let $f:=f_2$. Then $\tilde{f}=f_2$ by \eqref{eftilde}. We can argue as in the first inequality of \eqref{eq:aux888} to obtain
\begin{align*}
    \|f_2(u)-f_2(u_\cT)\|^2_{L^2(\Gamma_C)}
    &\leq
    4C_5^{2}\Big(C_3^{2}(\int_{\Gamma_C}e^{4C_3 u}\dx{s})^{1/2}+C_3^{2}(\int_{\Gamma_C}e^{4C_3 u_{\cT}}\dx{s})^{1/2}\\
&+C_4^{2}(\int_{\Gamma_C}e^{-4C_4 u}\dx{s})^{1/2}+C_4^{2}(\int_{\Gamma_C}e^{-4C_4 u_{\cT}}\dx{s})^{1/2}   \Big)\|u-u_\mathcal{T}\|_{L^4(\Gamma_C)}^2.
\end{align*}
Then an application of Lemma \ref{lem_ef2}, \eqref{stab_cont}, \eqref{stab_disc} and \eqref{eq:imb1} results in
\begin{align}\label{eq:555}
    \sum_{F\in\mathcal{F}_\cT(\Gamma_C)}h_F\|f_2(u)-f_2(u_\cT)\|^2_{L^2(F)}\leq C_{\rm{aux},6}\|u-u_\mathcal{T}\|_{H^1(\Om)}^2.
\end{align}
Here,
\begin{align*}
    C_{\rm{aux},6}:&=4\const{imb}{4}^{2}C_5^{2}\Big(C_3^{2}(2+2 |\Gamma_C|^{1/2}+e^{8C_3^{2}\Const{exp}C_{\rm{stb}}^2\|g\|_{L^2(\Gamma_A)}^2}+
e^{8C_3^{2}\Const{exp}C_{\rm{stb}}^2\|g\|_{L^2(\Gamma_A)}^2})\\
&+C_4^{2}(2+2 |\Gamma_C|^{1/2}+e^{8C_4^{2}\Const{exp}C_{\rm{stb}}^2\|g\|_{L^2(\Gamma_A)}^2}+e^{8C_4^{2}\Const{exp}C_{\rm{stb}}^2\|g\|_{L^2(\Gamma_A)}^2})\Big).
\end{align*}
This, together with \eqref{eq:444}, yields
\begin{align*}
    \|u-u_\mathcal{T}\|^2_{H^1(\Om)}+\mathrm{osc}_{\cT}^2(u_\cT)&\leq
    2\mathrm{osc}_{\cT}^2(v_\cT)+2C_{\rm{op}}\sum_{F\in\mathcal{F}_\cT(\Gamma_C)}h_F\|\tilde{f}(u)-\tilde{f}(v_\cT)\|^2_{L^2(F)}\\
    &+\Big(2\Const{op}+2\Const{op}\Const{cea}+\Const{cea}+2\Const{op}\Const{cea}\Const{aux,6}\Big)\|u-v_\mathcal{T}\|^2_{H^1(\Om)}.
\end{align*}
By taking  $C_{\rm{qs}}:=\max\{2,2C_{\rm{op}},2\Const{op}+2\Const{op}\Const{cea}+\Const{cea}+2\Const{op}\Const{cea}\Const{aux,6}\}$, we complete the proof.
\end{proof}
\begin{rem}
We refer to the square root of the left hand side of \eqref{opt_total} as the total error. Therefore, Lemma \ref{lem:lower-error} establishes the quasi-optimality of the solution $u_{\cT}$ in the sense of the total error. Note that the right hand side of \eqref{opt_total} can also be controlled by the total error:
\begin{equation*}
\begin{aligned}
\inf_{v_\cT\in V_\cT}\Big(\|u-v_\cT\|^2_{H^1(\Om)}&+\mathrm{osc}_\cT^2(v_\cT)
+\sum_{F\in\mathcal{F}_\cT(\Gamma_C)}h_F\|\tilde{f}(u)-\tilde{f}(v_\cT)\|^2_{L^2(F)}\Big)\\
&\leq(1+\const{aux}{6})\|u-u_\mathcal{T}\|^2_{H^1(\Om)}+\mathrm{osc}_{\cT}^2(u_\cT).
\end{aligned}
\end{equation*}
This estimate can be derived from \eqref{eq:555} directly.
\end{rem}
Next, we introduce the approximation class. Let $\mathbb{T}_{N}\subset \mathbb{T}$ be a subset consisting of all triangulation $\mathcal{T}\in\mathbb{T}$ satisfying
$\#\mathcal{T}-\#\mathcal{T}_{0}\leq N$. The approximation class $\mathbb{A}_{s}$ for
$0<s\leq m/2$ is defined by
\begin{equation*}
\mathbb{A}_{s}:=\big\{(u)\big|~|(u)|_{s}:=
\sup_{N>0}N^{s}\xi(N;u,\sigma)<+\infty\big\}
\end{equation*}
with
\[
    \xi(N;u,\sigma):=\displaystyle{\inf_{\mathcal{T}\in\mathbb{T}_{N}}\inf_{v_{\mathcal{T}}\in V_{\mathcal{T}}}}\Big(   \|u-v_{\mathcal{T}}\|_{H^1(\Omega)}^{2}+\mathrm{osc}_{\mathcal{T}}^{2}(v_{\cT})+\sum_{F\in\mathcal{F}_\cT(\Gamma_C)}h_F\|\tilde{f}(u)-\tilde{f}(v_\cT)\|^2_{L^2(F)}\Big)^{1/2}.
\]
The upper bound $m/2$ is attained for the uniform refinement.

Then we give the fundamental ingredients in the analysis, i.e., the optimal marking and cardinality of $\mathcal{M}_k$. We follow \cite{ffp} to derive the optimal marking that relates a strict error estimator reduction to D\"{o}rfler marking. This type of estimate was first given in \cite{stev1} for the Poisson equation with an $H^1$-norm reduction, and then extended in \cite{ckns} to the total error for linear elliptic problems. Below, we present a version in terms of the error estimator as in \cite{ffp}, the proof of which does not require the efficiency estimate in Theorem \ref{thm_eff}.
\begin{lem}[Optimal marking]\label{lem_optmar}
Suppose that the marking parameter $\theta$ in \eqref{marking} satisfies
\begin{equation}\label{mp_req}
\theta\in(0,1/(1+\Const{est}C_{drel})).
\end{equation}
Let $\cT_\ast\in\mathbb{T}$ be any refinement of $\cT\in\mathbb{T}$ and let $u_\cT\in V_\cT$, $u_{\cT_\ast}\in V_{\cT_\ast}$ be solutions
to problem \eqref{disvp} over $\cT$ and $\cT_\ast$, respectively.
For any $\delta>0$, define
\[\lambda:=({1-(1+C_{est}C_{drel})\theta})/{(1+\delta)}\in(0,1).\]
Assume further that
\begin{equation}\label{est_sred}
\eta_{\mathcal{T}_\ast}^2(u_{\cT_\ast})\leq\lambda
\eta_{\mathcal{T}}^2(u_{\cT}).
\end{equation}
Then there holds
\begin{equation*}
\eta_{\cT}^2(u_\cT,\cT\setminus\cT_\ast)\geq\theta\eta_{\cT}^2(u_\cT).
\end{equation*}
\end{lem}
\begin{proof}
We get by the estimate \eqref{eq:777} in Remark \ref{rem_estredLocalper}, \eqref{est_sred} and the discrete reliability estimate (Lemma \ref{lem_disrel}) that
\begin{align*}
\eta_{\cT}^2(u_\cT)&=\eta_{\cT}^2(u_\cT,\cT\setminus\cT_\ast)
+\eta_{\cT}^2(u_\cT,\cT\cap\cT_\ast)\\
&\leq \eta_{\cT}^2(u_\cT,\cT\setminus\cT_\ast)+(1+\delta)\eta^2_{\cT_\ast}(u_{\cT_\ast},\cT\cap\cT_\ast)
+\Const{est}\|u_{\cT_\ast}-u_\cT\|^2_{H^1(\Om)}\\
&\leq\eta_{\cT}^2(u_\cT,\cT\setminus\cT_\ast)+(1+\delta)\lambda\eta_{\cT}^2(u_\cT)+\Const{est}\|u_{\cT_\ast}-u_\cT\|^2_{H^1(\Om)}\\
&\leq(1+\Const{est}\Const{drel})\eta_{\cT}^2(u_\cT,\cT\setminus\cT_\ast)+(1+\delta)\lambda\eta_{\cT}^2(u_\cT).
\end{align*}
A direct calculation leads to
\begin{equation*}
\eta_{\cT}^2(u_\cT,\cT\setminus\cT_\ast)\geq\dfrac{1-(1+\delta)\lambda}{1+\Const{est}\Const{drel}}\eta_{\cT}^2(u_\cT)=\theta\eta_{\cT}^2(u_\cT).
\end{equation*}
This proves the assertion.
\end{proof}

\begin{lem}[Cardinality of $\mathcal{M}_k$]\label{lem_card}
Assume that condition \eqref{mp_req} holds. Let $u$ be the solution to problem \eqref{vp} and let $\{\cT_k,V_k,u_k\}$ be the sequence of meshes, finite element spaces and discrete solutions generated by Algorithm \ref{afem}. If $(u)\in\mathbb{A}_s$, then with $\lambda$  in Lemma \ref{lem_optmar}, there holds
\begin{equation*}
\#\mathcal{M}_k\leq\Big(C_{qs}(C_{rel}+1)/\lambda C_{eff}\Big)^{1/2s}|(u)|^{1/s}_s\left(\|u-u_k\|_{H^1(\Om)}^2+\mathrm{osc}^2_k(u_k)\right)^{-1/2s}.
\end{equation*}
\end{lem}
\begin{proof}
The proof is similar to \cite[Lemma 5.10]{ckns}. The assumption $(u)\in\mathbb{A}_{s}$ ensures that for
\begin{align}\label{eq:eps}
\varepsilon^2:=\lambda C_{\rm{eff}}\Big(\Const{qs}(C_{\rm{rel}}+1)\Big)^{-1}\left(\|u-u_k\|_{H^1(\Om)}^2+\mathrm{osc}^2_k(u_k)\right)
\end{align}
with a fixed $k\in\mathbb{N}_{+}$, there exist
a triangulation mesh $\mathcal{T}_\varepsilon\in\mathbb{T}$ and $v_{\varepsilon}\in V_{\mathcal{T}_\varepsilon}$ such that
\begin{equation}\label{lem_card01}
\#\cT_\varepsilon-\#\cT_0\leq|(u)|_s^{1/s}\varepsilon^{-1/s},\quad
\|u-v_{\varepsilon}\|^2_{H^1(\Om)}+\mathrm{osc}_{\cT_\varepsilon}^2(v_\varepsilon)
+\sum_{F\in\mathcal{F}_{\cT_{\varepsilon}}(\Gamma_C)}h_F\|\tilde{f}(u)-\tilde{f}(v_{\cT_{\varepsilon}})\|_{L^2(F)}^2
\leq \varepsilon^2.
\end{equation}
Let $\mathcal{T}_{\ast}$ be the smallest common refinement of $\mathcal{T}_{\varepsilon}$ and $\mathcal{T}_{k}$, i.e.,
$\mathcal{T}_{\ast}:=\mathcal{T}_\varepsilon\oplus\mathcal{T}_{k}$.
Then $\mathcal{T}_{\ast}$ is a refinement of $\mathcal{T}_\varepsilon$. An application of Theorem \ref{thm_eff} and \eqref{opt_total}, combining with the inequality  $\mathrm{osc}_{\cT_\ast}^2(v_{\varepsilon})\leq\mathrm{osc}_{\cT_\varepsilon}^2(v_\varepsilon)$ and \eqref{lem_card01}, leads to
\begin{align*}
C_{\rm{eff}}\eta_{\cT_{\ast}}^2(u_{\cT_{\ast}})&\leq\|u-u_{\cT_\ast}\|_{H^1(\Om)}^2
+\mathrm{osc}_{\cT_\ast}^2(u_{\cT_\ast})\\
&\leq \Const{qs}\Big(\|u-v_\varepsilon\|_{H^1(\Om)}^2
+\mathrm{osc}_{\cT_\varepsilon}^2(v_\varepsilon)
+\sum_{F\in\mathcal{F}_{\cT_{\varepsilon}}(\Gamma_C)}h_F\|\tilde{f}(u)-\tilde{f}(v_{\cT_{\varepsilon}})\|_{L^2(F)}^2\Big)
\leq \Const{qs}\varepsilon^2.
\end{align*}
This, together with Theorem \ref{thm_rel}, the inequality $\mathrm{osc}_{k}^2(u_{k})\leq\eta_{k}^2(u_{k})$ and \eqref{eq:eps}, gives
\[
\eta_{\cT_{\ast}}^2(u_{\cT_{\ast}})\leq\lambda\eta_{k}^2(u_{k}).
\]
Consequently, the subset $\mathcal{T}_{k}\setminus\mathcal{T}_{\ast}$ satisfies the D\"{o}rfler marking strategy owing to Lemma \ref{lem_optmar}. But the module MARK in Algorithm \ref{afem} selects a subset $\mathcal{M}_{k}\subset\mathcal{T}_{k}$ with minimal cardinality
such that the same property holds, which, together with Lemma 3.7 in \cite{ckns}, implies
\begin{equation}\label{lem_card02}
\#\mathcal{M}_k\leq\#\cT_{\ast}-\#\cT_k\leq\#\cT_\varepsilon-\#\cT_0.
\end{equation}
Therefore, the assertion readily follows from \eqref{lem_card01} and \eqref{lem_card02}.
\end{proof}

Now we can establish the quasi-optimality of Algorithm \ref{afem}.
\begin{thm}\label{thm_opt}
Let condition \eqref{mp_req} hold. Let $u$
be the solution to problem \eqref{vp} and
$\{\mathcal{T}_{k},V_{k},u_{k}\}$ be the sequence of
meshes, finite element spaces and discrete solutions generated by
Algorithm \ref{afem}. If $(u)\in\mathbb{A}_{s}$, then there
holds
\begin{equation*}
\|u-u_{k}\|_{H^1(\Om)}^{2}+\mathrm{osc}_{k}^{2}(u_k)\leq
C_{qopt}|(u)|_{s}^{2}(\#\mathcal{T}_{k}-\#\mathcal{T}_{0})^{-2s},
\end{equation*}
where $C_{qopt}$ depends on $C_{\mathcal{T}_{0}}$, $m$, $\mu$ and $\beta$ in Theorem
\ref{thm_conv} but is independent of $s$ or $u$.
\end{thm}
\begin{proof}
Let $M:=(\Const{qs}(C_{\rm{rel}}+1)/\lambda C_{\rm{eff}})^{1/2s}|(u)|^{1/s}_s$. By
\eqref{stevlemma} and Lemma \ref{lem_card}, we deduce
\begin{equation*}
\#\mathcal{T}_k-\#\mathcal{T}_0\leq C_0\sum_{j=0}^{k-1}\#\mathcal{M}_j\leq
C_0M\sum_{j=0}^{k-1}\Big(\|u-u_j\|^2+\mathrm{osc}^2_j(u_j)\Big)^{-1/2s}.
\end{equation*}
Since the oscillation term \eqref{globalosc} is dominated by the global error, we can obtain from
Theorem \ref{thm_eff} and Lemma \ref{lem_equiv} that
\begin{equation*}
\begin{split}
c_{\rm{equ}}\|u&-u_j\|_{H^1(\Om)}^2+\beta\mathrm{osc}^2_j(u_j)\leq \mathcal{J}(u_j)-\mathcal{J}(u)+\beta\eta^2_j(u_j)\\
&\leq(C_{\rm{equ}}+\beta C_{\rm{eff}}^{-1})(\|u-u_j\|_{H^1(\Om)}^2+\mathrm{osc}^2_j(u_j)).
\end{split}
\end{equation*}
Now Theorem \ref{thm_conv} implies that
\begin{equation*}
\mathcal{J}(u_k)-\mathcal{J}(u)+\beta\eta_k^2(u_k)\leq\mu^{k-j}
\big(\mathcal{J}(u_j)-\mathcal{J}(u)+\beta\eta_j^2(u_j)\big), \text{ for } 0\leq j\leq k-1.
\end{equation*}
Now collecting the last three estimates, we arrive at
\begin{align*}
\#\mathcal{T}_k-\#\mathcal{T}_0&\leq C_0M(C_{\mathrm{equ}}+\beta C_{\mathrm{eff}}^{-1})^{1/2s}\sum_{j=0}^{k-1}(\mathcal{J}(u_j)-\mathcal{J}(u)+\beta\eta_j^2(u_j))^{-1/2s}\\
&\leq C_0M(C_{\mathrm{equ}}+\beta C_{\mathrm{eff}}^{-1})^{1/2s}(\mathcal{J}(u_k)-\mathcal{J}(u)+\beta\eta_k^2(u_k))^{-1/2s}\sum_{j=1}^{k}\mu^{j/2s}\\
&\leq C_0C_{\theta}M(C_{\mathrm{equ}}+\beta C_{\mathrm{eff}}^{-1})^{1/2s}(1/\min(c_{\mathrm{equ}},\beta))^{1/2s}(\|u-u_k\|^2+\mathrm{osc}^2_k(u_k))^{-1/2s}
\end{align*}
with $C_\theta:=\mu^{1/2s}/(1-\mu^{1/2s})$ bounding the geometric series. Raising this to the $s$-th power and noting $C_0^s\leq C_0^{m/2}$, $C_\theta^{s}\leq 1/(1-\mu^{1/m})^{s}\leq 1/(1-\mu^{1/m})^{m/2}$ give the desired estimate.
\end{proof}

\section{Numerical results}
Now we present two numerical tests using Algorithm \ref{afem} with affine elements. The implementation of the
algorithm is based on \cite{FunkenPW11}. In the experiments, $\Omega$ is an L-shaped domain $\Omega=[-1,1]^2\backslash ([0,1]\times[-1,0])$,
and $\sigma=1$. The initial mesh  $\mathcal{T}_0$ is a uniform triangulation of the
domain, cf. Fig. \ref{fig:initial_mesh}. At the $k^{\text{th}}$ adaptive iteration with triangulation $\mathcal{T}_k$ and $k=0,1,2,\cdots$, we employ the Newton method to
obtain the corresponding solution $u_k$ (of the nonlinear system). Specifically, we take the
initial guess to be linear interpolation from the previous mesh, i.e.,
\begin{equation*}
u^{(0)}_{k}=
\left\{
\begin{aligned}
&0, &&k=0 \\
&\mathcal{I}_{k}(u_{k-1}), &&k>0
\end{aligned}
\right.
\end{equation*}
with $\mathcal{I}_{k}:V_{k-1}\to V_{k}$ being the linear interpolation operator.
The stopping criterion for the Newton iteration is
\[
\|u_{k}^{(n)}-u_{k}^{(n-1)}\|_{H^1(\Omega)}\leq \epsilon
\]
with $n$ denoting the Newton iteration number, and $\epsilon$ the prescribed accuracy. In the adaptive algorithm, we take $\epsilon=10^{-7}$.
\begin{figure}[htb!]
\centering
\includegraphics[trim={0.5cm 0.5cm 0.5cm 0.5cm},clip,width = .3\textwidth, height  = .3\textwidth]{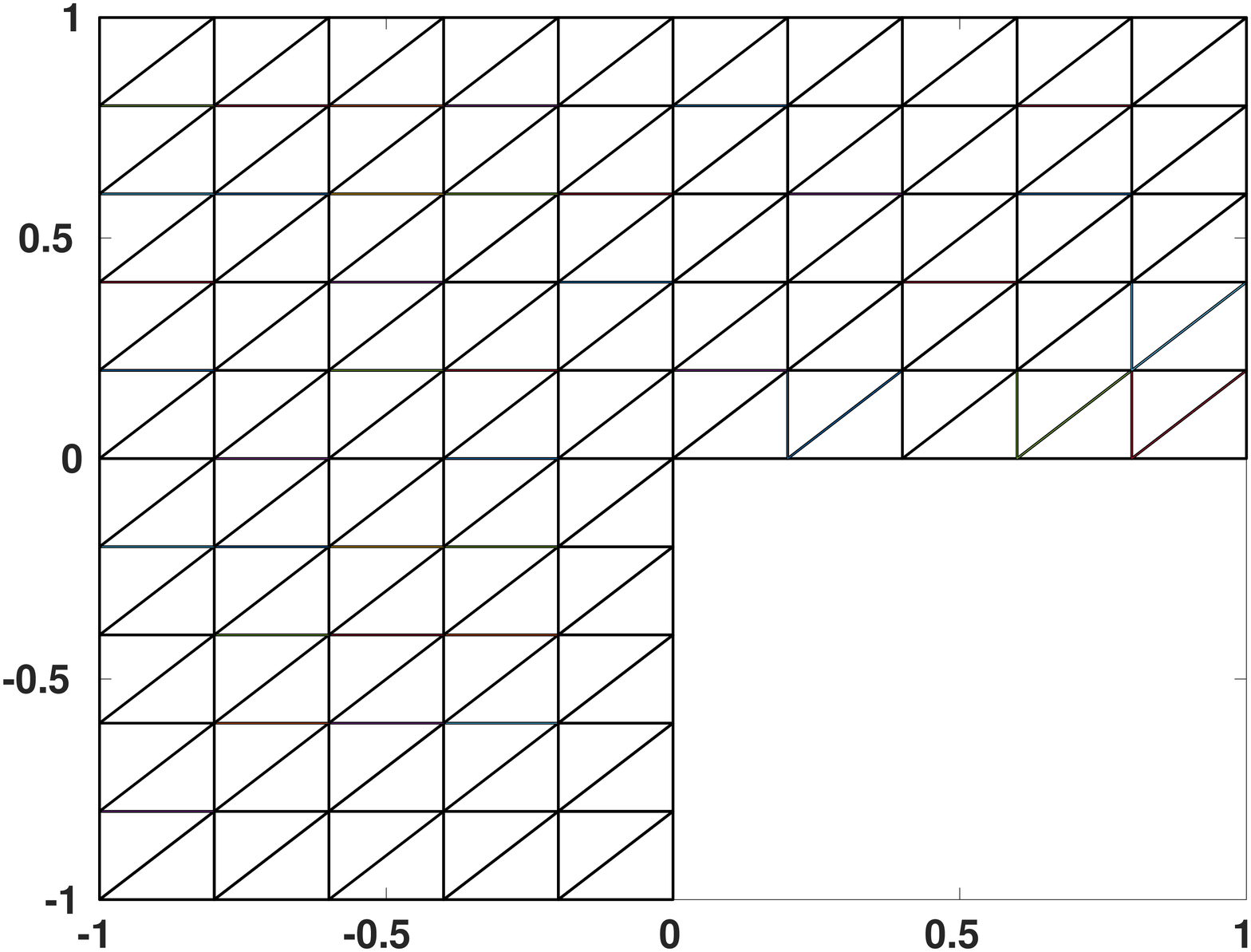}
\caption{Initial uniform triangle mesh with a mesh size $h=0.2$.}
\label{fig:initial_mesh}
\end{figure}
Algorithm \ref{afem} is terminated once the sum of error indicators
$\sum_{T\in \mathcal{T}_k}\eta_k(u_k,T)^2$ falls below a pre-specified
threshold tolerance $\tau$. We take $\tau=10^{-3}$ for both examples below.
After obtaining the adaptive solution $u_k$, we check whether the stopping condition is satisfied. If not, a refinement is
carried out for those with large error indicators.

In our simulation, the D\"{o}rfler bulk criterion is used to mark elements for refinement, 
i.e., given $\theta\in (0,1)$, we look for the
minimal set $\mathcal{M}_k\subset \mathcal{T}_k$ satisfying \eqref{marking}.
In the experiment we show results with $\theta=0.1$ and $\theta=0.3$. To refine the mesh, we apply the newest vertex bisection (NVB) refinement \cite{FunkenPW11} and bisect all
edges of the elements in $\mathcal{M}_k$ to get a finer mesh $\mathcal{T}_{k+1}$.
Alternatively, we can mark only one edge of those elements in $\mathcal{M}_k$ for a refined mesh.
Note that a smaller $\theta$ yields a more adaptive mesh and a larger iteration number.

\begin{example}\label{example:1}
In the first example, denote $\Gamma_0=\emptyset$, $\Gamma_C$ as the left boundary, and $\Gamma_A$ as the rest of the
boundary. We take $g(x,y)=x^2+y^2$ and $f(u)=u+u^3$.
\end{example}

Due to the nonlinearity of the problem, the exact solution is not available. Hence, we use the solution on a very fine
uniform mesh with a mesh size $h=1/2000$ as the reference solution (and analogously, the Newton method is employed with a
much smaller tolerance $\epsilon=10^{-11}$). The solution on an adaptive mesh is shown in Fig. \ref{fig:solutions}(a).
Since the solution singularity is localized around the re-entrant corner of the domain and the corners where the
boundary condition changes, the adaptive algorithm properly refines these regions. In Fig. \ref{fig:solutions_eg1_conv}(a),
we observe a convergence rate $O(N^{-0.51})$ for the error estimator, which agrees well with the convergence rate
$O(N^{-0.54})$ in the $H^1(\Omega)$-norm error of the adaptive solution from Theorem \ref{thm_opt}, numerically verifying the reliability of
the estimator. The adaptive algorithm is more efficient than the uniform refinement. Fig.
\ref{fig:solutions_eg1_conv}(b) displays the convergence history with a larger parameter $\theta=0.3$. We
obtain a smaller iteration number but a larger degrees of freedom over each refinement. The convergence rate
of the error estimator and the $H^1(\Om)$-error are $O(N^{-0.50})$ and $O(N^{-0.53})$, respectively.

\begin{figure}[htb!]
\centering
\begin{subfigure}[b]{0.5\textwidth}
\includegraphics[height =0.3\textheight, width =1\textwidth]{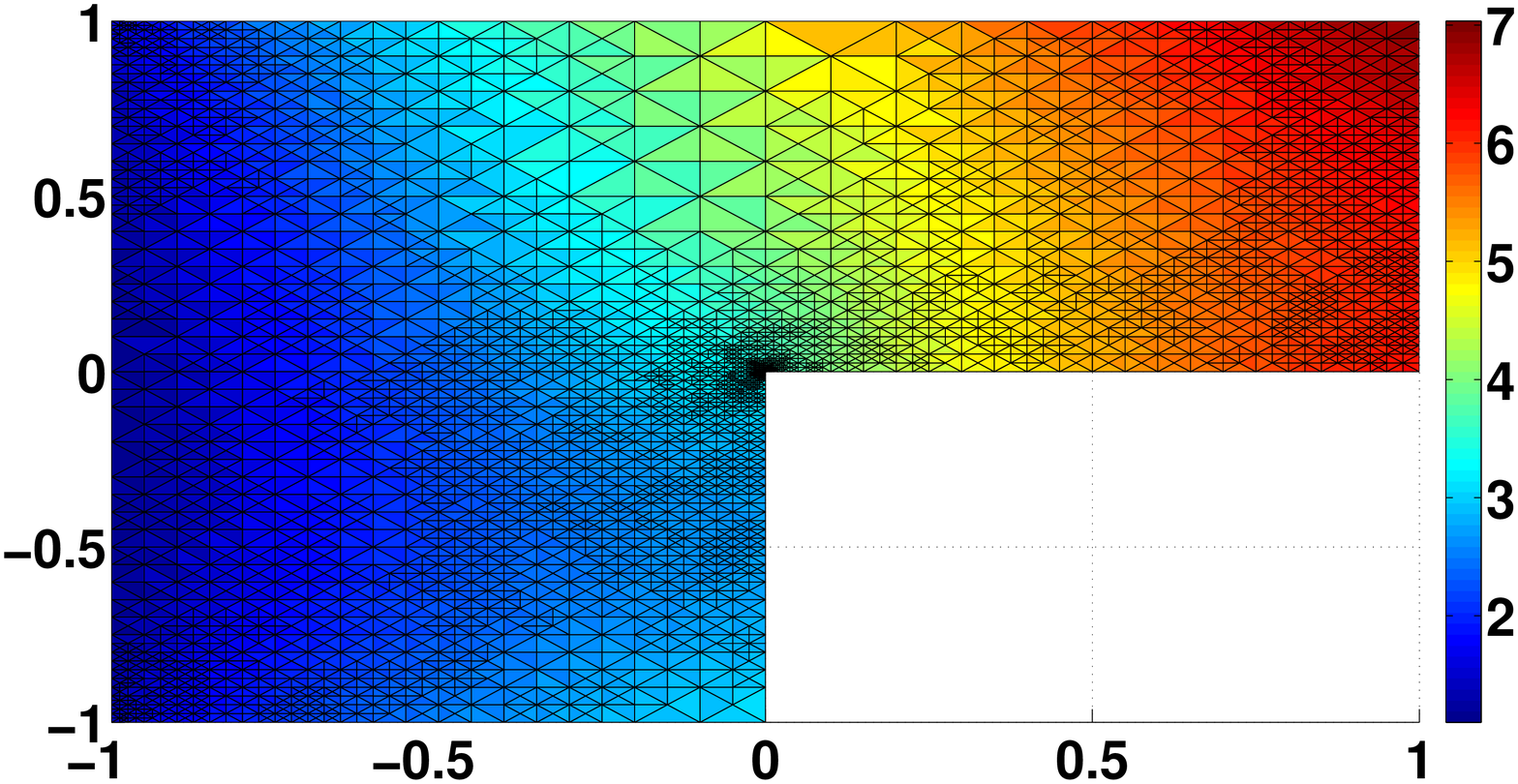}
\caption{Example \ref{example:1}}
\end{subfigure}
\begin{subfigure}[b]{0.49\textwidth}
\includegraphics[height =0.3\textheight, width =1\textwidth]{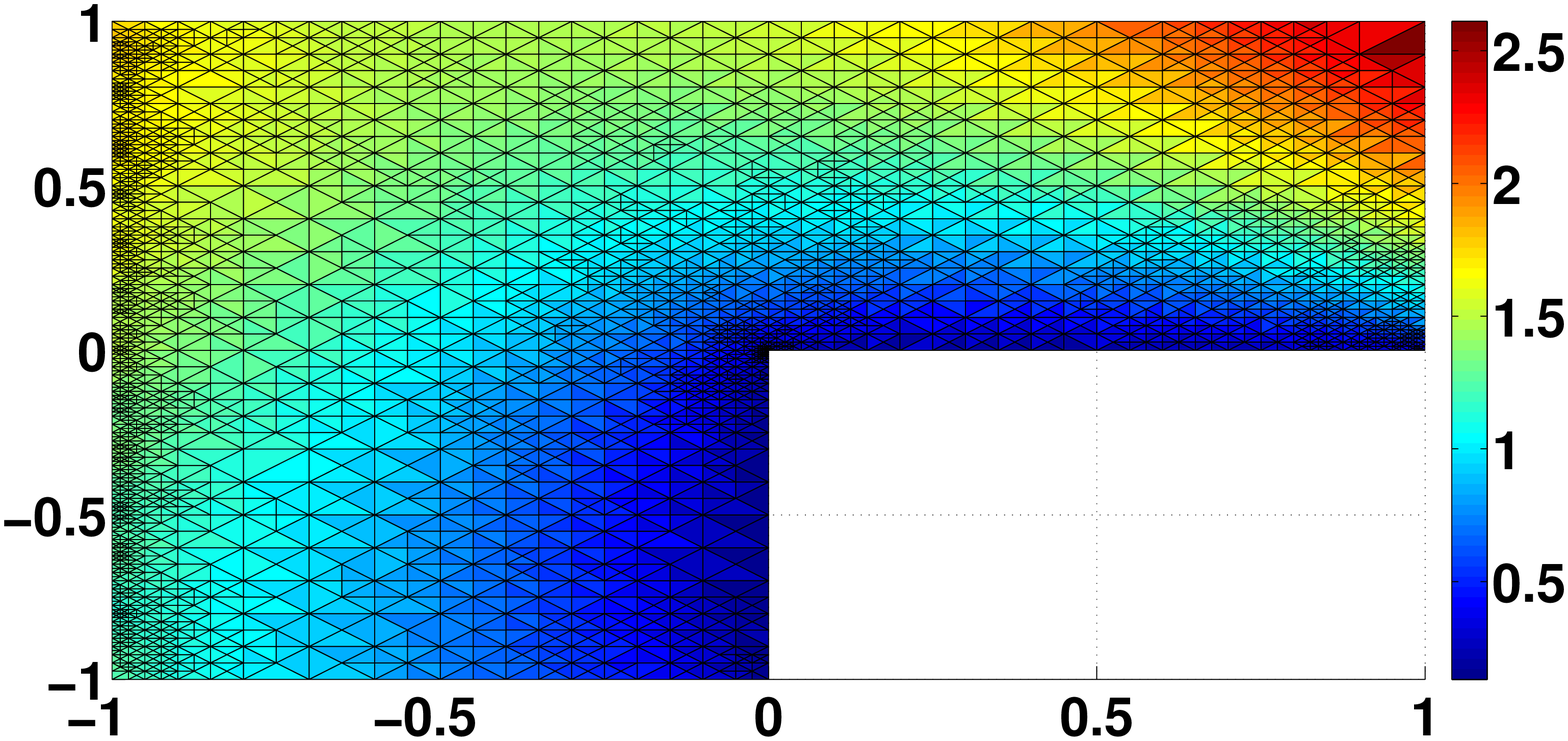}
\caption{Example \ref{example:2}}
\end{subfigure}
\caption{The adaptive solution with $\theta=0.1$ and $\tau=10^{-3}$. Panel (a)
gives the adaptive solution with $k=25$, dof=3248 and the
$H^1(\Omega)$-relative error is $6.48\%$; Panel (b) shows the adaptive solution
with $k=22$, dof=3070 and the $H^1(\Omega)$-relative error of $6.27\%$.}
\label{fig:solutions}
\end{figure}

\begin{figure}[htb!]
\centering
\begin{subfigure}[b]{0.48\textwidth}
\includegraphics[trim={10.5cm 0 10.5cm 0},clip,width =1\textwidth]{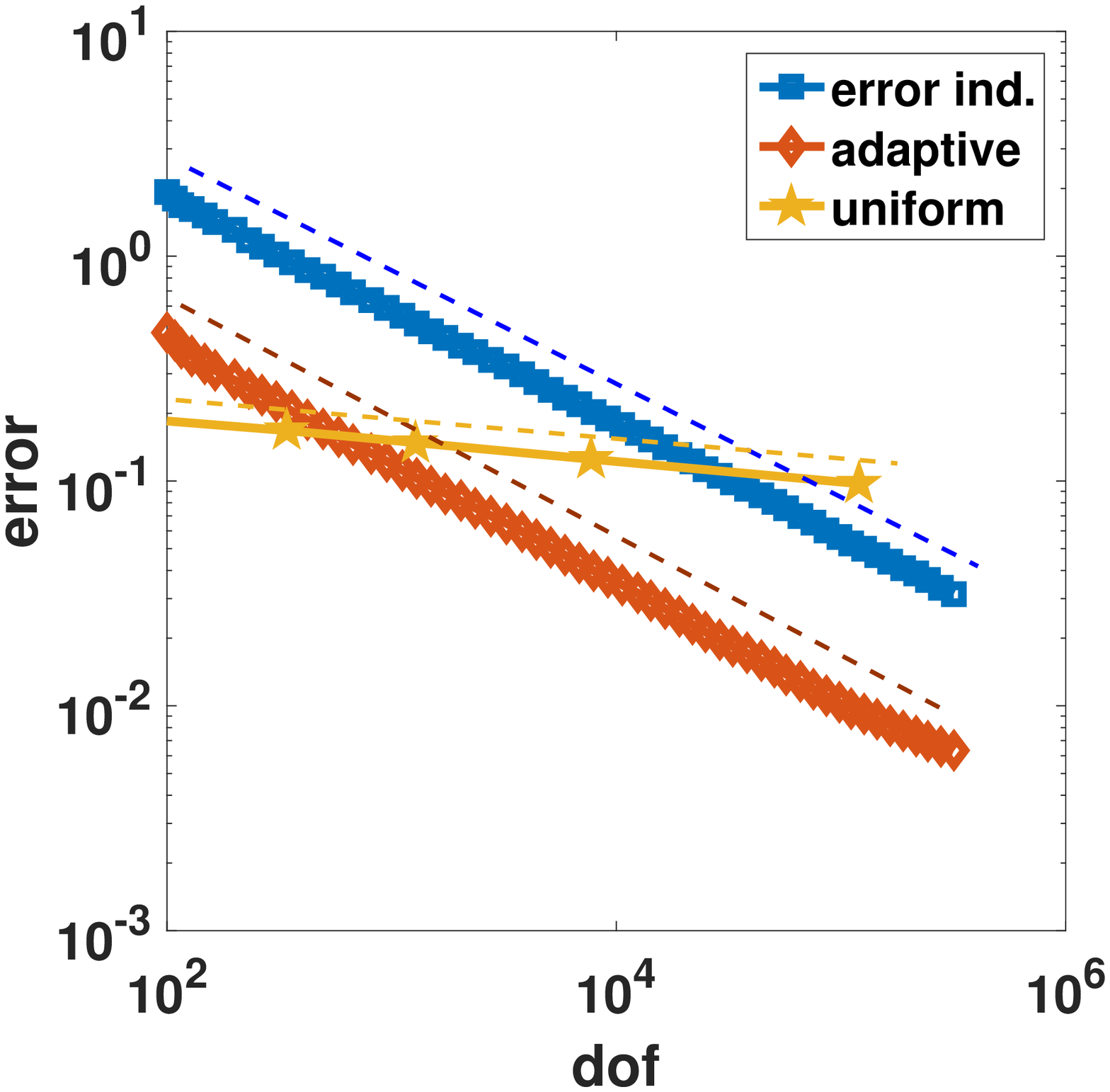}
\caption{$\theta=0.1$, iteration: 58}
\end{subfigure}
\begin{subfigure}[b]{0.48\textwidth}
\includegraphics[trim={10.5cm 0 10.5cm 0},clip,width =1\textwidth]{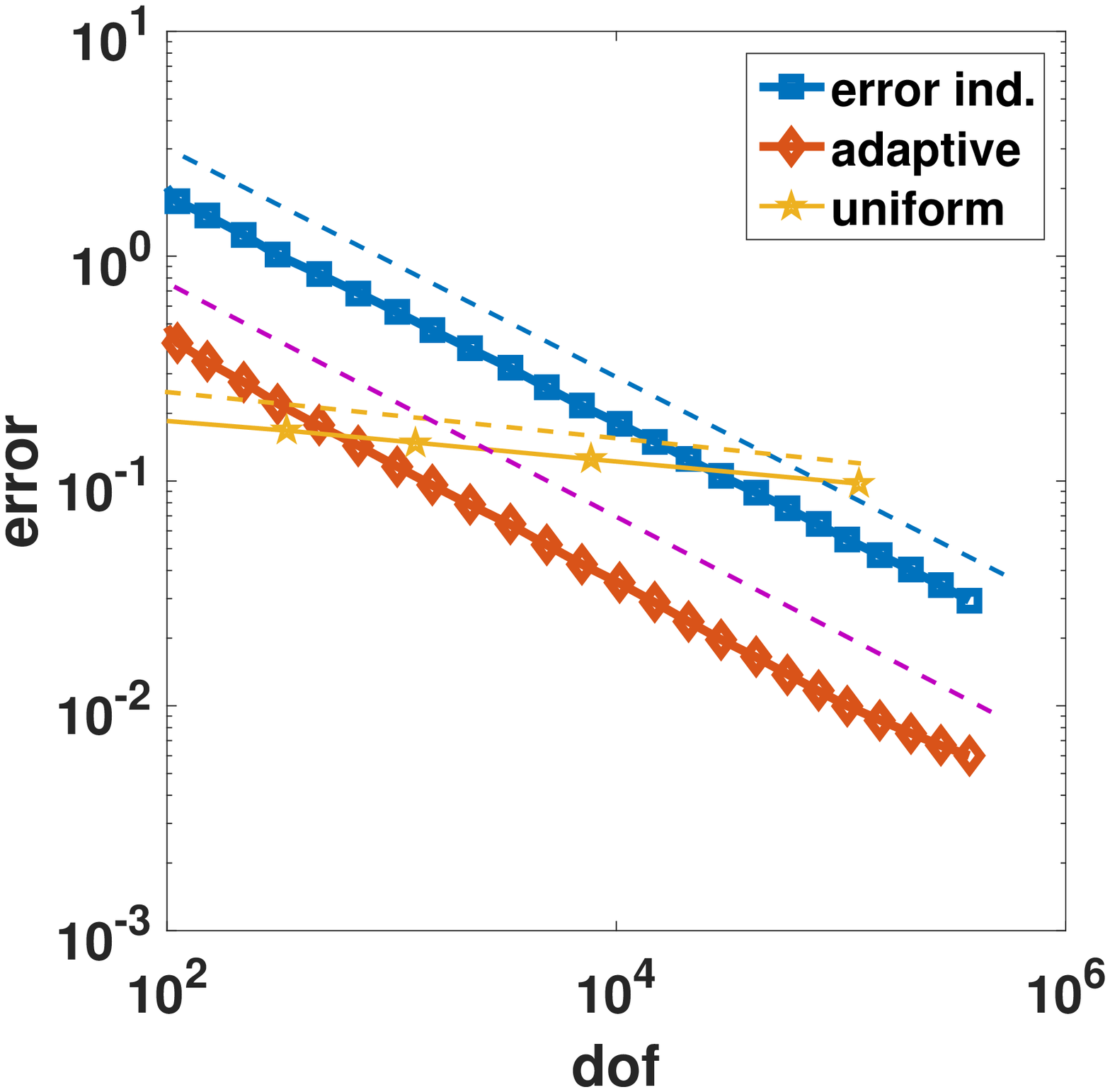}
\caption{$\theta=0.3$, iteration: 24}
\end{subfigure}
\caption{Error estimator and $H^1(\Omega)$-error versus dof. with $\tau=10^{-3}$
for Example \ref{example:1}. In Panel (a), the slopes of the dashed lines are -0.51, -0.54
and -0.09 for the indicator, the adaptive refinement $H^1(\Omega)$-error and the uniform
refinement $H^1(\Omega)$-error, respectively, and in Panel (b), the slopes of
the dashed lines are -0.50, -0.53 and -0.09, for the estimator, the adaptive
refinement $H^1(\Omega)$-error and the uniform refinement $H^1(\Omega)$-error, respectively.}
\label{fig:solutions_eg1_conv}
\end{figure}

In the second example, we consider the oscillatory boundary condition.
\begin{example}\label{example:2}
In this example, let $\Gamma_0=\emptyset$, $\Gamma_C$ be the boundary segments with the re-entrant corner
and $\Gamma_A$ as the rest of the boundary. We take
\begin{equation*}
g(x,y)=
\left\{
\begin{aligned}
\sin(20y), &\quad \text{on } \Gamma_1=\{-1\}\times[-1,1],\\
\sin(x)+\cos(y),&\quad \text{on } \Gamma_{A}\setminus\Gamma_1.\\
\end{aligned}
\right.\quad\mbox{and}\quad f(u)=e^{5u}-e^{-5u}.
\end{equation*}
\end{example}

\begin{figure}[htb!]
\centering
\begin{subfigure}[b]{0.48\textwidth}
\includegraphics[trim={10.5cm 0 10.5cm 0},clip,width=1\textwidth]{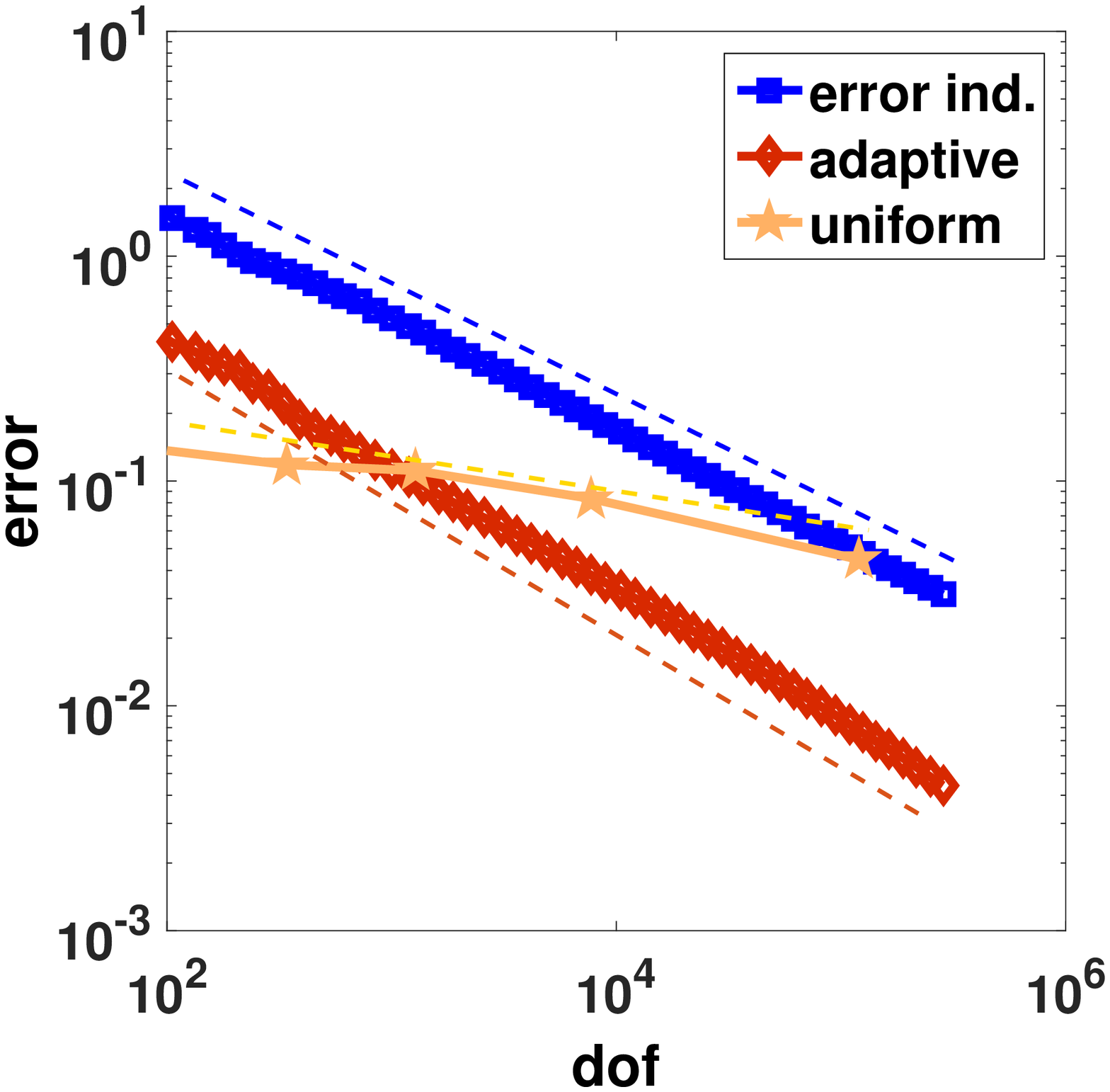}
\caption{$\theta=0.1$, iteration: 53.}
\end{subfigure}
\begin{subfigure}[b]{0.45\textwidth}
\includegraphics[trim={10.5cm 0 10.5cm 0},clip,width=1\textwidth]{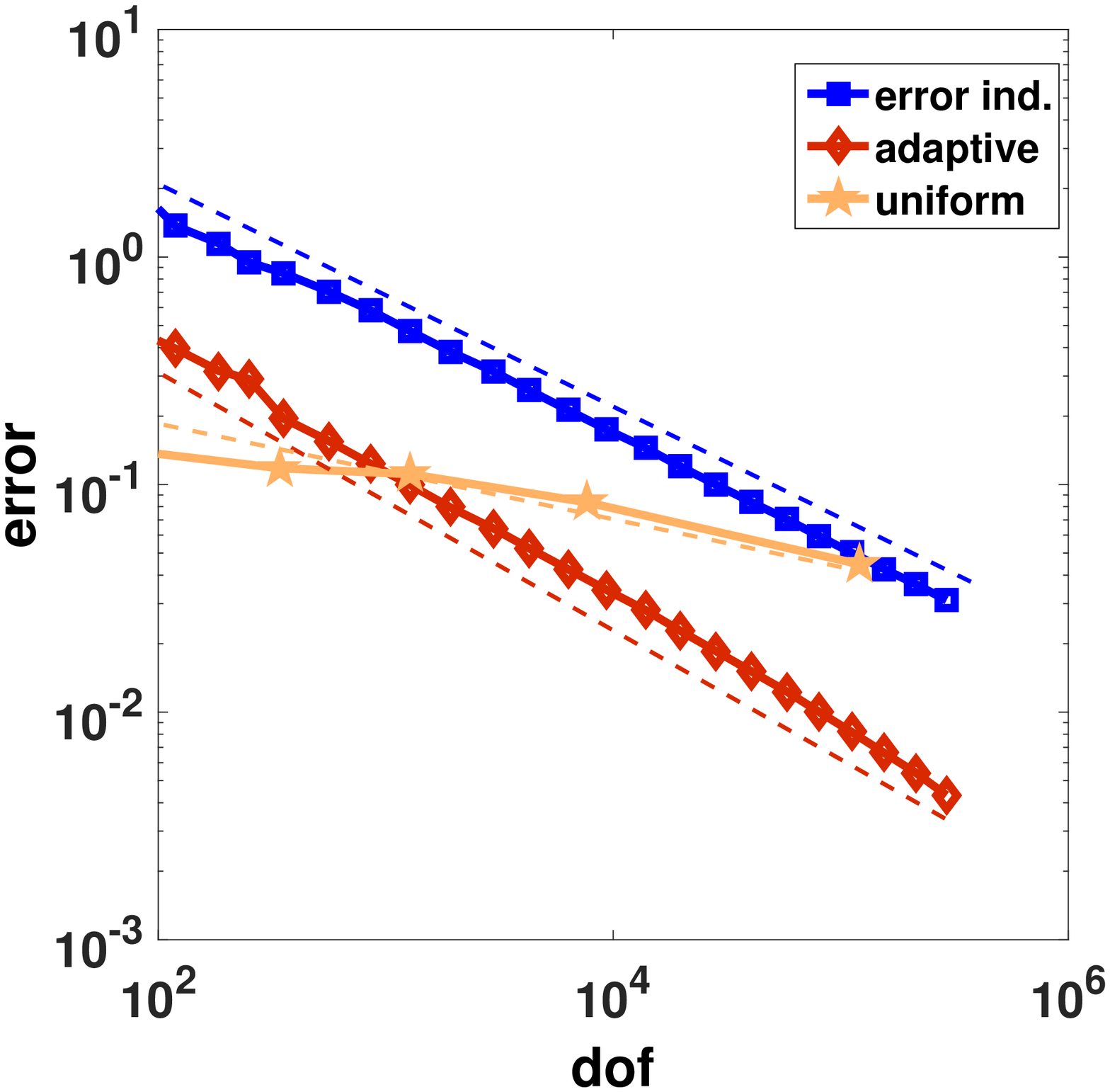}
\caption{$\theta=0.3$, iteration: 22.}
\end{subfigure}
\caption{Error estimator and $H^1(\Omega)$-error versus dof. 
with $\tau=10^{-3}$ for Example \ref{example:2}. In Panel (a), the slopes of the dashed lines are -0.50, -0.56 and -0.11, for the estimator, the adaptive refinement $H^1(\Omega)$-error and the uniform refinement $H^1(\Omega)$-error, respectively, and in Panel (b),
the slopes of the dashed lines are -0.51, -0.55 and -0.11, for the estimator, the adaptive refinement
$H^1(\Omega)$-error and the uniform refinement $H^1(\Omega)$-error, respectively.}
\label{fig:solutions_eg2_conv}
\end{figure}

In Example \ref{example:2}, the numerical solution on a fine mesh with a mesh size $h=1/2000$ and parameter $\epsilon=10^{-11}$
is taken to be the reference solution. The numerical results for the example are shown in Figs. \ref{fig:solutions}(b)
and \ref{fig:solutions_eg2_conv}. Due to the oscillatory boundary data, the region close to the left boundary requires
adaptive refinement, in addition to the re-entrant corner and the corners where the boundary condition changes.
On a very coarse mesh, the oscillatory boundary data is not properly resolved, which leads to a slower decay at the
beginning. Nonetheless, as the adaptive procedure proceeds, the convergence of the algorithm is fairly
steady, with the estimator decay rate $O(N^{-0.50})$ and the $H^1(\Omega)$ convergence rate $O(N^{-0.56})$ for
$\theta=0.1$. We observe similar convergence rates for $\theta=0.3$ from Fig. \ref{fig:solutions_eg2_conv}(b).

\section{Concluding remark}

In this paper, for a 2D variational problem governed by a linear diffusion equation
and a nonlinear boundary condition, we have analyzed an adaptive finite element method based on a
residual-typed {\em a posteriori} error estimator and the D\"{o}rfler marking. We established a quasi-optimal
decay rate in terms of the number of elements for the algorithm, which is confirmed by the
numerical experiments. One natural question is to extend the analysis to the 3D case.


\begin{thebibliography}{99}

\bibitem{ao}
M. Ainsworth and J. T. Oden, \textit{A Posteriori Error Estimation
in Finite Element Analysis}, Pure and Applied Mathematics,
Wiley-Interscience, New York, 2000.

\bibitem{br}
I. Babu\v{s}ka and W. Rheinboldt, \textit{Error estimates for
adaptive finite element computations}, SIAM J. Numer. Anal., 15
(1978), 736-754.


\bibitem{bdk}
L. Belenki, L. Diening and C. Kreuzer, \textit{Optimality of an adaptive finite
element method for the p-Laplacian equation}, IMA J. Numer. Anal., 32 (2012), 484-510.

\bibitem{bdd}
P. Binev, W. Dahmen and R. DeVore, \textit{Adaptive finite element
methods with convergence rates}, Numer. Math., 97 (2004), 219-268.

\bibitem{Brenner&Scott}
S. C. Brenner and L. R. Scott, \textit{The Mathematical Theory of Finite Element Methods}, Texts in Applied Mathematics, 3rd Edition, Springer, New York, 2008.

\bibitem{cfpp}
C. Carstensen, M. Feischl, M. Page and D. Praetorius, {\it Axioms of adaptivity},
Comp. Math. Appl., 67 (2014), 1195-1253.

\bibitem{ckns}
J. M. Cascon, C. Kreuzer, R. H. Nochetto and K. G. Siebert,
\textit{Quasi-optimal convergence rate for an adaptive finite
element method}, SIAM J. Numer. Anal., 46 (2008), 2524-2550.


\bibitem{ciarlet}
P. G. Ciarlet, \textit{Finite element methods for elliptic
problems}, North-Holland, Amsterdam, 1978.


\bibitem{dk}
L. Diening and C. Kreuzer, \textit{Linear convergence of an adaptive finite element method for the p-Laplacian
equation}. SIAM J. Numer. Anal., 46 (2008), 614-638.

\bibitem{dks}
L. Diening, C. Kreuzer and R. Stevenson, \textit{Instance optimality of the adaptive maximum Strategy}, Found. Comput. Math., 16 (2016), 33-68.

\bibitem{dor}
W. D\"{o}rfler, \textit{A convergent adaptive algorithm for
Poisson's equation}, SIAM J. Numer. Anal., 33 (1996), 1106-1124.

\bibitem{ffp}
M. Feischl, T. F\"{u}hrer and D. Praetorius, \textit{Adaptive FEM with optimal convergence
rates for a certain class of non-symmetric and possibly non-linear problems}, SIAM J. Numer. Anal., 52 (2014), 601-625.

\bibitem{FunkenPW11}
S. A. Funken, D. Praetorius and P. Wissgott, \textit{Efficient implementation of adaptive $P1$-FEM in Matlab}, Comput. Methods Appl. Math., 11 (2011), 460-490.


\bibitem{gmz2}
E. M. Garau, P. Morin and C. Zuppa, \textit{Convergence of an adaptive Ka\v{c}anov FEM for quasi-linear problems}, Appl. Num. Math., 61 (2011), 512-529.

\bibitem{gmz3}
E. M. Garau, P. Morin and C. Zuppa, \textit{Quasi-optimal convergence rate of an AFEM for quasi-linear problems of monotone type}, Numerical Mathematics: Theory, Methods and Applications, 5 (2012), 131-156.

\bibitem{hs}
L. S. Hou and W. Sun, \textit{Optimal positioning of anodes for cathodic protection}, SIAM Cont and Opt., 34 (1996), 855-873.

\bibitem{ht}
L. S. Hou and J. C. Turner, \textit{Analysis and finite element approximation of an optimal control problem in electrochemistry with current density controls}, Numer. Math.,
71 (1995), 289-315.

\bibitem{koss}
I. Kossaczky. \textit{A recursive approach to local mesh refinement
in two and three dimensions}, J. Comp. Appl. Math., 55 (1995),
275-288.


\bibitem{LiXu}
G. Li and Y. Xu, \textit{A convergent adaptive finite element method for cathodic protection}, Comput. Methods Appl. Math., 17 (2017), 105-120.

\bibitem{mitc}
W. F. Mitchell, \textit{A comparison of adaptive refinement techniques for elliptic problems}. ACM Trans. Math. Software, 15 (1989), 326--347.


\bibitem{nsv}
R. H. Nochetto, K. G. Siebert and A. Veeser, \textit{Theory of
adaptive finite element methods: an introduction}, Multiscale,
Nonlinear and Adaptive Approximation (R. A. DeVore and A. Kunoth,
Eds), Springer, New York, 2009, 409-542.

\bibitem{sz90}
L. R. Scott and S. Zhang, \textit{Finite element interpolation of
nonsmooth functions satisfying boundary conditions}, Math. Comp., 54
(1990), 483-493.



\bibitem{stev1}
R. Stevenson, \textit{Optimality of a standard adaptive finite element method}, Found.
Comput. Math., 7 (2007), 245-269.

\bibitem{stev2}
R. Stevenson, \textit{The completion of locally refined simplicial partitions created by bisection}, Math. Comp., 77 (2008), 227-241.

\bibitem{traxler}
C. Traxler, \textit{An algorithm for adaptive mesh refinement in $n$
dimensions}, Computing, 59 (1997), 115-137.

\bibitem{ver}
R. Verf\"{u}rth, \textit{A Posteriori Error Estimation Techniques for Finite Element Methods
}, Oxford University Press, Oxford, 2013.


\end{thebibliography}
\end{document}